\documentclass[]{interact}

\usepackage{epstopdf}
\usepackage{subfigure}

\usepackage{natbib}
\bibpunct[, ]{(}{)}{;}{a}{}{,}
\renewcommand\bibfont{\fontsize{10}{12}\selectfont}

\theoremstyle{plain}
\newtheorem{theorem}{Theorem}
\newtheorem{lemma}[theorem]{Lemma}

\usepackage{parskip}

\theoremstyle{definition}

\theoremstyle{remark}

\newcommand{\vp}[1]{\textcolor{blue}{#1}}

\newcommand{\m}{\text{master}}

\usepackage{graphicx}

\usepackage[colorlinks=true,linkcolor=black,citecolor=black]{hyperref}
\usepackage{caption}
\usepackage{enumitem} 
\usepackage[]{algorithm} 
\usepackage[noend]{algpseudocode}
\usepackage{multirow}

\usepackage[table,xcdraw]{xcolor} 
\usepackage{lscape} 

\usepackage{color}
\usepackage[backgroundcolor=white, bordercolor=blue, linecolor=blue, textsize=tiny]{todonotes}

\begin{document}


\title{Computationally-Efficient Decomposition Heuristic for the Static Traffic Assignment Problem}

\author{
\name{Venktesh Pandey\textsuperscript{a}\thanks{Corresponding Author: Venktesh Pandey. Email: vpandey@ncat.edu} and Priyadarshan N. Patil\textsuperscript{b}}
\affil{\textsuperscript{a}Assistant Professor, North Carolina Agricultural and Technical State University, North Carolina, USA; \textsuperscript{b}Graduate Research Assistant, The University of Texas at Austin, Texas, USA}
}

\maketitle

\begin{abstract}
Applications such as megaregional planning require efficient methods for solving traffic assignment problems (TAPs) on large-scale networks. We propose a decomposition heuristic that generates approximate TAP solutions by partitioning the complete network into subnetworks which are solved in parallel and use an iterative-refinement algorithm for improving the network partitions.  A novel network transformation and three-stage algorithm are also proposed to solve a constrained shortest path problem as a subproblem of the heuristic. Experiments on various networks show that the heuristic can generate 15.1--67.8\% computational savings in finding solutions with initial relative gap of 0.02. The performance benefits of the proposed heuristic when warmstarting standard TAP algorithms are demonstrated with an average computational savings of 10--35\% over a TAP solver without warmstarting. 
\end{abstract}


\section{Introduction}

Traffic assignment problems (TAPs), which assign vehicles to network routes as the last step of the four-step planning process, are well studied. The classical time-independent or static version of TAP takes as input a demand matrix (number of travelers going from one zone to the other) and the network supply side characteristics (transportation network and travel-time functions on links as a function of number of travelers), and solves for flows on network links such that the Wardrop equilibrium condition is satisfied. Static TAP is used ubiquitously for transportation planning purposes such as select link analysis, scenario comparison, and flow assignment in the four-step planning process. With advances in computing power, several efficient algorithms have been proposed to solve the static TAP. This research is motivated by the need to improve the computational efficiency of solving TAPs on large-scale networks.

The first motivation is planning for megaregions. Megaregions cross state and political boundaries and are characterized by trade and infrastructural connections across a much larger area. For example, the north-eastern megaregion spans the states of Connecticut, Rhode Island, New York, Massachusetts, Pennsylvania, Delaware, New Jersey and Virginia, and contains as much as 17\% of the current US population. With increasing intra-megaregion trade and traffic, current statewide or citywide models are not appropriate for the megaregion scale creating a need for computationally-efficient TAP algorithms for these networks.

The second motivation is rooted in applications that require multiple TAP solutions, such as scenario analysis and network design problems\citep{gokalp2021post, patil2022budget}. These applications can benefit significantly from faster TAP algorithms. Furthermore, for such applications we often desire only an approximate solution to the TAP, characterized by a higher \textit{relative gap} value than an optimal solution at which the relative gap is zero~\citep{patriksson2015traffic, patil2020convergence}.

Lately, network decomposition algorithms have been proposed for TAP~\citep{jafari2017decomposition, raadsen2018aggregation}. These algorithms decompose a large network into subnetworks which can be solved in parallel and offer a potential to outperform algorithms that solve TAP on the full network (referred to as centralized algorithms) by making use of multiple machine cores. Furthermore, decentralized computations also offer robustness against centralized machine failures~\citep{ehsanPhDThesis}. However, scaling these algorithms for more than two network partitions and handling subnetwork interactions is challenging. In this article, we propose a decomposition-based heuristic that generates approximate solutions to the TAP in faster computation time than the centralized algorithms.

The contributions of this research include the following: (a) we propose a novel network transformation and three-stage algorithm to solve a restricted shortest path (SP) problem as a subproblem of the heuristic,  
(b) we propose partitioning algorithms that work for general networks without obvious geographical partitions (such as a natural partition created by a river) and discuss their impacts on solving TAP for large scale networks, which is an improvement over prior work~\citep{yahia2018network}, and (c) we propose an iterative refinement approach for network partitioning that improves the computational efficiency of the heuristic. The rest of the article is organized as follows. Section \ref{sec:formulationlitreview} presents the formulation and background for TAP and its solution methods. Sections~\ref{sec:methods} and \ref{sec:otherDetails} present the methodology and the heuristic, and discuss the implementation detail.  Section~\ref{sec:results} reports experimental results and findings on different large-scale test networks. Lastly, Section \ref{sec:conclusions} concludes the article and presents potential future extensions.

\section{Formulation and Literature Review}
\label{sec:formulationlitreview}
This section provides the convex programming formulation for static TAP and then discusses advances in TAP solution methods. 

\subsection{Formulation}
\label{sec:formulation}

Let $G=(N,A)$ be the directed network defined by set of nodes $N$ and set of arcs $A$. For each node $i\in N$, let $\Gamma_i$ denotes the set of all outgoing arcs and $\Gamma^{-1}_i$ denote the set of all incoming arcs. A subset of $N$, denoted by $Z$, indicates the set of centroids where demand originates or terminates. 
The demand from origin zone $r$ to destination zone $s$ is denoted by $d^{rs}$, and the origin-destination (OD) pair is denoted by $<r,s>$. The set of all demand values is recorded in an OD matrix $D$, defined as $D=\{d^{rs}~\vert~r\in Z, s\in Z \}$. Sets $G$ and $D$ are inputs for static TAP. 

The output of static TAP is the steady state flow $x_{ij}$ on each link $(i,j) \in A$, starting at node $i$ and terminating at node $j$. The travel times on each link as a function of flow is given by a link performance function, which is commonly assumed to be an increasing and a convex function of link flow. One commonly used function is the Bureau of Public Roads equation~\citep{BPR}, stated as follows:
\begin{equation}
\begin{centering}
t_{ij}(x_{ij}) = t_{ij}^0(1+\alpha(x_{ij}/u_{ij})^\beta), 
\end{centering}
\end{equation}

where $t_{ij}^0$ is the free-flow travel time (i.e., the travel time on the link with zero congestion), $u_{ij}$ is the capacity of the link, and $\alpha$ and $\beta$ are shape parameters calibrated to the link. If calibration is not possible, the default values are $\alpha$ = 0.15 and $\beta$ = 4.

Let $\pi$ denote a path, and $h_\pi$ denote the number of drivers who choose the path. The set of all paths connecting zones $r$ and $s$ is denoted by $\Pi^{rs}$. Link and path flows can be tied together using the following equation:
\[x_{ij} = \sum_{r\in Z}\sum_{s\in Z}\sum_{\pi\in \Pi^{rs}} \delta_{ij}^\pi h_\pi,\]
where $\delta_{ij}^\pi$ is an indicator variable which takes the value 1 if link $(i,j)$ is used by path $\pi$, and 0 otherwise. Similarly, path travel times $c^\pi$ are defined as the sum of travel times of the links included on the path. Thus,
\[c^\pi =\sum_{(i,j)\in A} \delta_{ij}^\pi t_{ij}. \]
Traffic equilibrium is a steady state of the network where no driver can get an advantage by unilaterally changing their decision. More formally, the principle of user equilibrium states ``Every used route connecting an origin and destination has equal and minimal travel cost". This is also known as Wardrop's first principle~\citep{wardrop}. This principle makes a few assumptions including (a) the assumption of driver's rationality where each driver wants to choose the path between their origin and their destination with the least travel cost, and (b) the assumption of perfect information, where drivers have perfect knowledge of link travel times. With these assumptions, the TAP for user equilibrium can be formulated as a convex program as follows:
\begin{subequations}
\begin{align}
\min_{\mathbf{x},\mathbf{h}} \qquad & \sum_{(i,j) \in A} \int_0^{x_{ij}} t_{ij}(x) dx & \\
\text{subject to:} & & \nonumber \\
    x_{ij} &= \sum_{\pi \in \Pi : (i,j) \in \pi} h_\pi & \forall (i,j) \in A \label{eq:beckmannLinkFlow}\\
    \sum_{\pi \in \Pi^{rs}} h_\pi &= d_{rs} & \forall(r,s) \in Z^2 \label{eq:beckmannPathFlows}\\
    h_\pi &\geq 0 & \forall \pi \in \Pi \label{eq:beckmannNonNegativity}
\end{align}
\label{eq:beckmann}
\end{subequations}

This convex program formulation was first proposed by Martin Beckmann~\citep{Beckmann}. The constraints denote the relation between link flows and path flows (Equation~\eqref{eq:beckmannLinkFlow}), flow conservation across all paths for each OD pair (Equation~\eqref{eq:beckmannPathFlows}), and non-negativity of path flows (Equation~\eqref{eq:beckmannNonNegativity}). Assuming that link travel time functions are continuous guarantees existence of a solution and that these functions are strictly increasing guarantees uniqueness for the convex program. More details on the convex program formulation in Equation~\eqref{eq:beckmann} and its variants can be found in standard texts such as \cite{patriksson2015traffic} and \cite{blubook_vol1_v085}. This formulation has allowed for conventional convex programming approaches to be applied to TAP, as well as 
development of specialized TAP solution algorithms that are discussed next.


\subsection{TAP solution methods}
\label{sec:litreview}
Specialized iterative methods have been developed to solve the convex program in Equation~\eqref{eq:beckmann} efficiently by exploiting properties of TAP. These methods broadly fall into three classes: link-based, path-based and origin-based. Historically, link-based algorithms were the first to be studied, owing to low memory requirements. These methods update the link flow vector in order to strictly reduce the objective function. The method of successive averages algorithm~\citep{powell1982convergence} and Frank-Wolfe (FW) algorithm~\citep{frank1956,leblanc1975} are the two most widely used link based algorithms. Researchers have tried to improve upon these methods by modifying the search direction~\citep{patriksson2015traffic, daneva, mitradjieva2013, holmgren2014}. Link-based methods are easy to implement but are slow and inefficient for larger instances. Additionally, the solution does not include path information. 

Path-based methods address these issues by tracking used path sets, taking advantage of increased computing power. For each OD pair, the algorithms track a used path set and equilibrate flows within the used path set. Two widely known algorithms are gradient projection~\citep{jaya} and projected gradient~\citep{florian}. Some other algorithms in this class are disaggregate simplicial decomposition algorithm~\citep{larsson1992simplicial}, path based FW~\citep{chen2002}, conjugate gradient projection~\citep{lee2003}, slope-based methods~\citep{kumar2010,kumar2014}, and greedy path-based approach~\citep{xie2018greedy}. However, path-based algorithms have significant redundancy in path storage, due to overlap in multiple paths for an OD pair. 

This led to development of origin-based methods which keep track of a set of bushes (a connected acyclic subnetwork rooted at an origin). Knowing the bush structure allows for easy decomposition into path flows. Dial's Algorithm B~\citep{dial2006path} allows for flow equilibration only between the shortest and longest path in a bush, while Bar-gera's OBA~\citep{bar2002origin} shifts flows on multiple paths in a bush. \citet{nie2010class} compares multiple Newton-type bush-based algorithms. \citet{bar2010traffic} proposes TAPAS, an algorithm using flow shifts between paired alternative segments, rather than flow shifts over entire paths. Some other algorithms in this space are local user equilibrium algorithm~\citep{gentile2014local} and i-TAPAS~\citep{xie2016}.

These approaches solve TAP for the entire network. With nonlinear computational complexity for TAP, decomposing and solving multiple smaller problems can lead to better performance. \cite{jafari2017decomposition} decompose the network into smaller subnetworks and alternate solving the master network (using network contractions) and subnetworks till convergence. \cite{raadsen2018aggregation} focuses on creating an alternate lossless representation of the network and performing computational optimization on this representation using an integrated disaggregation-aggregation framework. In this article, we build off the geoegraphical-decomposition model proposed in \cite{jafari2017decomposition}.

Computational improvements have also been sought by parallelization of sub-routines such as shortest path subroutines~\citep{kumar1991scalability,florian2001applications}, or OD pair equilibration in Gradient projection~\citep{chen2020parallel}. Usage of contraction hierarchies~\citep{buchhold2018real,buchhold2019real} or efficient shortest path methods~\citep{ziliaskopoulos1997design} have also been proposed. Each of these efficiencies can be implemented as a complementary dimension to the geographical partitioning implemented in our proposed heuristic.

One last consideration is the usage of `warm-starting' to reduce computation time. Warm-starting has seen some prior usage for traffic assignment algorithms. First suggested by Dial~\citep{dial1999network2}, a feasible traffic assignment solution was proposed as a substitute to shortest path assignment. \citet{dial2006path} also found that warm-starting with a TAP solution reduced computation time for a slightly modified trip table. \citet{peeta2011post} show the benefits of warm-starting TAP algorithms on multiple networks, observing faster convergence. The benefits of warm-starting dynamic TAP using static TAP solutions have also been observed in multiple studies~\citep{levin2011utilizing,Nezamuddin2011,levin2015warm}.

The common motivation for these methods has been the reduction of computation time. The improvements are benchmarked on large-scale networks, including megaregions. In this article, we contribute towards this motivation by developing decomposition-based algorithms (as introduced in Section \ref{sec:methods}). We also exploit the usefulness of the proposed heuristic to generate solutions for efficiently warmstarting TAP, as discussed later.

\subsection{Convergence of TAP solution algorithms}

Traffic assignment algorithms provide link flows (and in some cases, path flows) as output. Relevant metrics such as total system travel time (TSTT) are calculated from the link flows. Given the current link flow vector $\mathbf{x}$, TSTT is  defined as the sum of travel times across all travelers and is calculated using the following formula:
\begin{equation}
    TSTT(\mathbf{x}) = \sum_{(i,j) \in A} t_{ij}(x_{ij})x_{ij}.
\end{equation}
Traffic assignment algorithms terminate based on a predetermined convergence criterion. The convergence criterion indicates proximity to the optimal solution. There are various criteria that are used in practice, such as relative gap, average excess cost or average/total link reduced cost~\citep{bar2010traffic}. \citet{patil2020convergence} provide some guidance on appropriate level of relative gap values required for various network metrics to stabilize. There are several definitions of relative gap in common practice; we use the following one.  Let $\kappa_{rs}(\mathbf{x})$ denote the travel time on the shortest path from origin $r$ to destination $s$ using the link travel times corresponding to $\mathbf{x}$.  The shortest path travel time (SPTT) equals TSTT if all vehicles were on shortest paths (as required by the user equilibrium condition):
\begin{equation}
    SPTT(\mathbf{x}) = \sum_{<r,s> \in Z^2} \kappa_{rs} (\mathbf{x}) d^{rs}.
\end{equation}
\noindent The relative gap (RG) can then be defined as: 
\begin{equation}
    \text{gap}(\mathbf{x}) = TSTT(\mathbf{x})\ -\ SPTT(\mathbf{x})\,
\end{equation}
\begin{equation}
    \text{RG}(\mathbf{x}) = \frac{\text{gap}(\mathbf{x})}{SPTT(\mathbf{x})} = \frac{TSTT(\mathbf{x})}{SPTT(\mathbf{x})}-1\,.
\end{equation}

\section{Methodology}
\label{sec:methods}
\subsection{DSTAP vs DSTAP-Heuristic}
\label{subsec:dstapvsheuristic}
\citet{jafari2017decomposition} proposed a decomposition algorithm for static traffic assignment problem (DSTAP) where a full network is replaced by three components: (a) geographically separate regions called subnetworks, (b) a master network that retains the origins, destinations, and boundary nodes in each subnetwork, and (c) artificial links that approximate the interactions between connecting nodes where the link parameters are estimated using TAP sensitivity analysis.  Figure~\ref{fig:visualDstap}(b) shows the DSTAP master network and subnetworks obtained after decomposing the full network in Figure~\ref{fig:visualDstap}(a). Artificial links in the master network approximate interactions within a subnetwork between an origin/destination node and a boundary node of a subnetwork. On the other hand, artificial links for subnetworks approximate the paths connecting boundary nodes of a subnetwork that traverse other subnetworks.

\begin{figure}[H]
\noindent\makebox[\textwidth]{%
\includegraphics[scale=0.35]{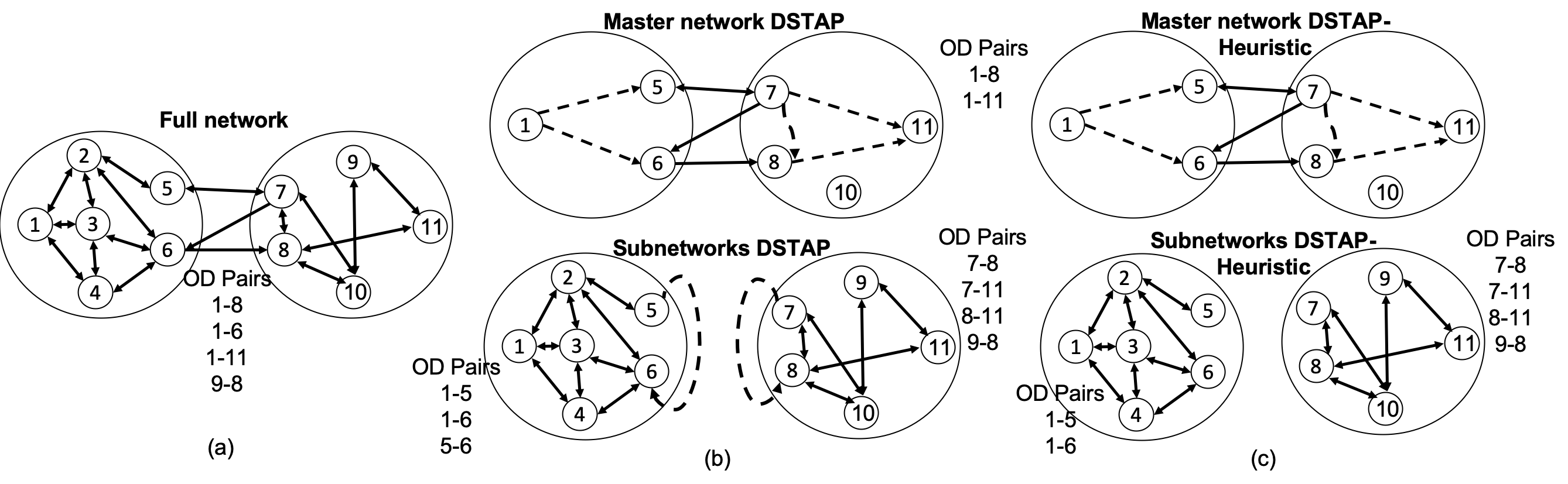}}
\centering
\caption{(a) Full network, (b) DSTAP networks, and (c) DSTAP-Heuristic networks}
\label{fig:visualDstap}
\centering
\end{figure}

The DSTAP procedure is shown in Algorithm~\ref{algo:dstap}. In each DSTAP iteration for the network shown in Figure~\ref{fig:visualDstap}(b), the master network TAP is solved first where demand between OD pairs $1\rightarrow8$ and $1\rightarrow11$ are loaded on paths in the master network. Next, the flows on each master network artificial link is used to update the demand on subnetworks; for example, flow obtained on artificial link $(1,5)$ is added to the demand between OD pair [1,5] in the left subnetwork. Next, the subnetworks are solved in parallel, and the obtained solution is used to find the updated artificial link parameters for the next iteration using a bush-based sensitivity analysis~\citep{jafari2016improved}. Last, the flows on each used path between an OD pair in the master network and subnetworks are mapped to the full network by identifying corresponding paths that replace an artificial link. For example, the flow on path $1\rightarrow 5\rightarrow 7 \rightarrow 11$ will be mapped to the full network by concatenating the paths connecting 1 and 5 in subnetwork 1, with link $(5,7)$, followed by paths connecting $7$ and $11$ in subnetwork 2. \cite{jafari2017decomposition} showed that the DSTAP iterative procedure is guaranteed to converge to the optimal TAP solution with additional computational savings for large-scale networks.

    \begin{algorithm}[]
		\caption{DSTAP procedure for solving TAP~\citep{jafari2017decomposition}}
		\label{algo:dstap}
		\begin{algorithmic}[]
			\State \textbf{Inputs}: network file, trips file, networks partitions as subnetworks
			\State \textbf{Step 1}: Set up --- validate partitions for connectivity, create master networks and subnetworks, generate master network artificial links and corresponding artificial origin-destination (OD) pairs inside subnetworks. 
			\State \textbf{Step 2}: Optimize
			\While {relative gap on the full network $>$ threshold AND iteration number $<$ desired maximum iterations}
				\State Solve master network to a desired gap
				\State Update the demand on subnetwork artificial OD pairs
				\State Solve the subnetworks in parallel with the updated demand to a desired gap
				\State Update the parameters for master network artificial links and subnetwork artificial links using TAP sensitivity analysis
			\EndWhile
			
			\State \textbf{Outputs}: Final full network gap and link flows/costs
		\end{algorithmic}
	\end{algorithm}

Implementing DSTAP for megaregions and other large-scale networks is advantageous as it distributes the computations over subnetworks, which can each be solved in parallel. However, the number of artificial links generated can pose a computational bottleneck. In the Austin network tested by \cite{jafari2017decomposition}, there were approximately 74000 artificial links in the master network and 3000 artificial links across the subnetworks (more details are presented in Table~\ref{tab:dstapVSheuristic}).

As observed, the number of master network artificial links is significantly high and one might expect their inclusion to pose a computational burden for DSTAP. However, that is not the case. This is because, for the case with 2 subnetworks, the used paths in the master network connecting an OD pair only includes 2 artificial links: one connecting origin to the boundary node of origin's subnetwork, and the one connecting boundary node of destination's subnetwork to the destination. Additionally, artificial links in the master network are essential in modeling the interaction between interior and boundary nodes. Thus, master network artificial links are not a primary challenge.

On the other hand, subnetwork artificial links (SALs) such as (dashed) links (5,6) and (7,8) in subnetworks in  Figure~\ref{fig:visualDstap}(b), which represent all the paths connecting the two nodes completely outside the given subnetwork, are particularly challenging due to three reasons. 

First, at any DSTAP iteration, a path replacing an SAL may itself include SALs in other subnetwork, which adds an unnecessary complexity to the calculation of sensitivity parameters and mapping of flow on the full network. For example, consider the network shown in Figure~\ref{fig:issue1DSTAP}(a) where the used paths at equilibrium connecting nodes $a$ and $d$ in subnetwork 1 are shown. If we replace interactions outside the subnetwork with artificial links, then artificial link $(a,d)$ is modeled by conducting sensitivity analysis of OD pair $(e,h)$ in subnetwork 2, whose paths include another artificial link $(f,g)$. This pattern where a SAL may replace other (possibly multiple) SALs create computational issues of SAL parameters' from one iteration to the next and also while mapping flow from SALs to full network paths.

\begin{figure}[h]
\centering
\subfigure[]{%
\resizebox*{4.5cm}{!}{\includegraphics{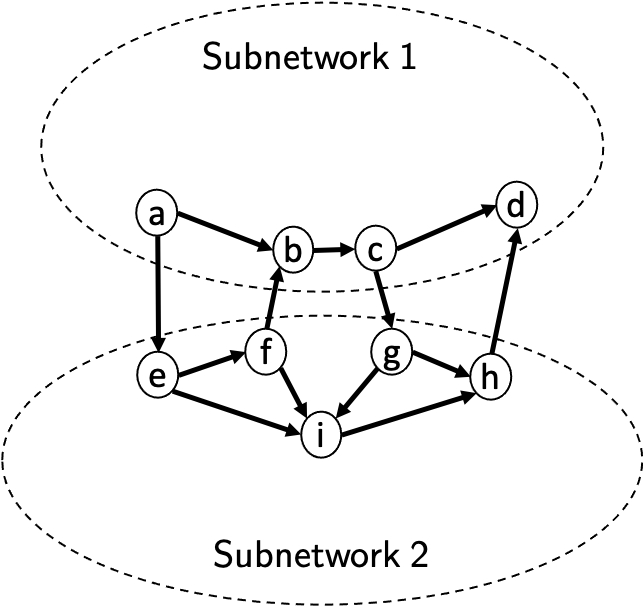}}}\hspace{2pt}
\subfigure[]{%
\resizebox*{4.5cm}{!}{\includegraphics{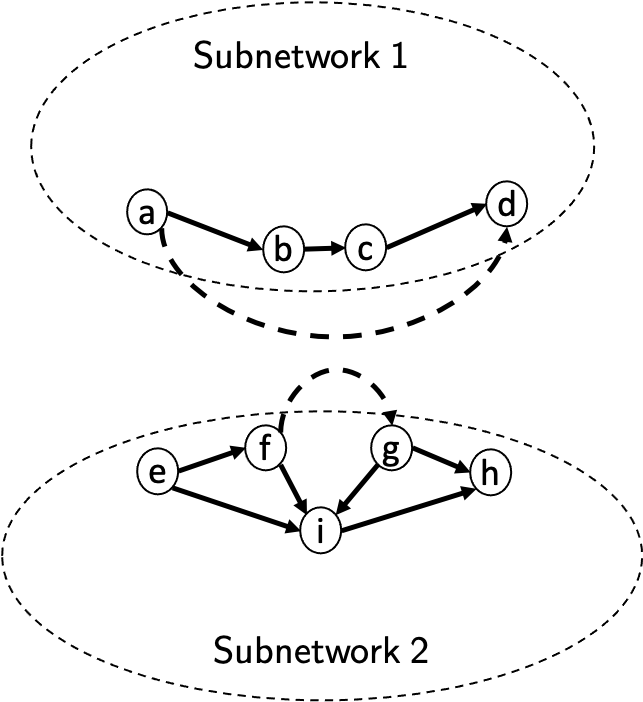}}}
\caption{Issue 1 where a SAL (a,d) in subnetwork 1 includes a path that includes another SAL $(f,g)$ in subnetwork 2 creating computational burden with flow mapping and computing SAL parameters using sensitivity analysis.} \label{fig:issue1DSTAP}
\end{figure}

Second, as discussed in~\cite{jafari2017decomposition}, modeling SALs for more than two subnetworks requires approximating a part of the master network in all other subnetworks, which adds additional computational burden by increasing the number of links in each subnetwork. For example, in the figure shown below, we observe that to model outside-the-network interactions between nodes $a_1$ and $a_2$ in subnetwork 1, we require including 12 extra links (with regional and master-network artificial links included), which can make the process of solving each subnetwork cumbersome.

\begin{figure}[h]
\centering
\subfigure[]{%
\resizebox*{5.5cm}{!}{\includegraphics{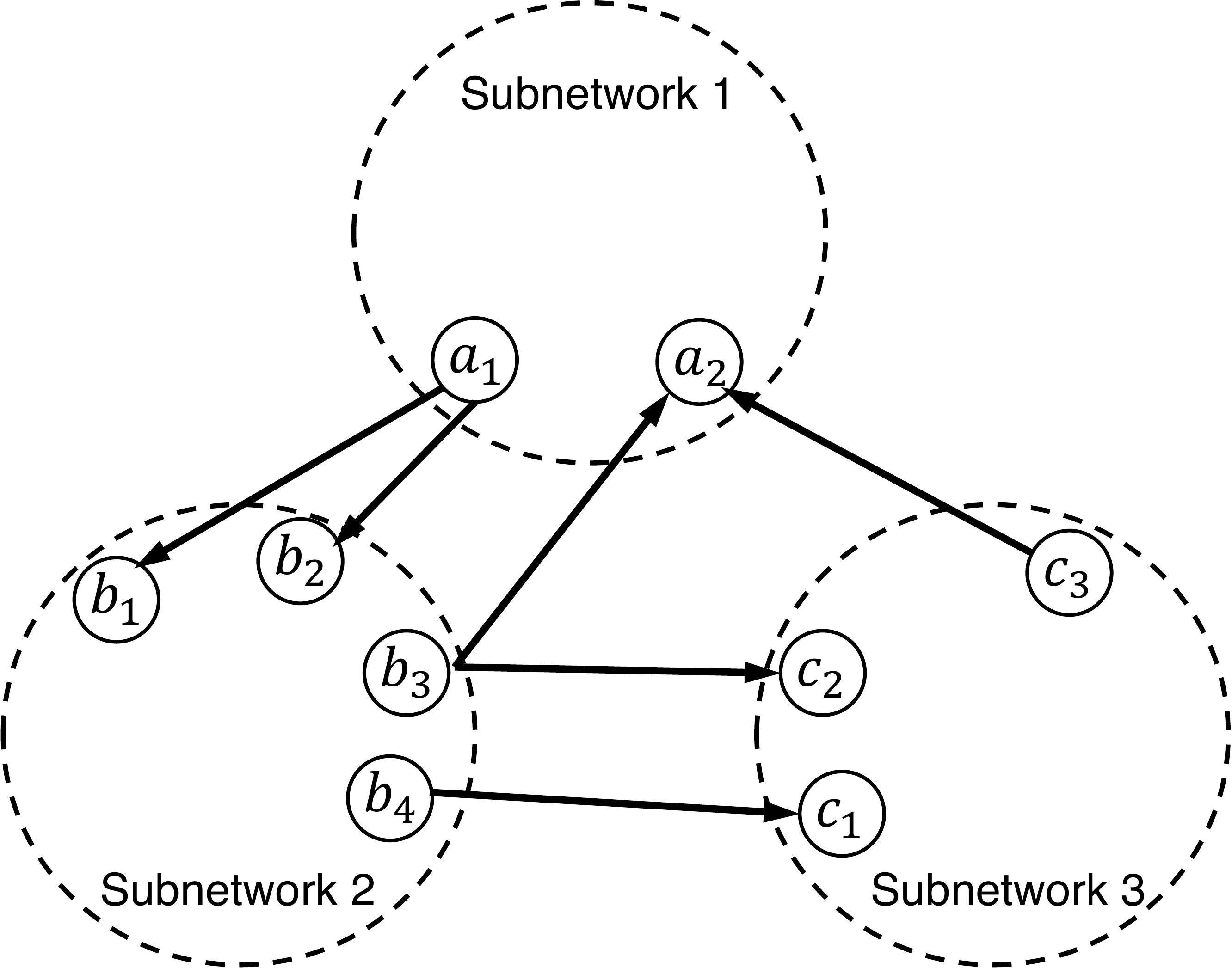}}}\hspace{1pt}
\subfigure[]{%
\resizebox*{4.7cm}{!}{\includegraphics{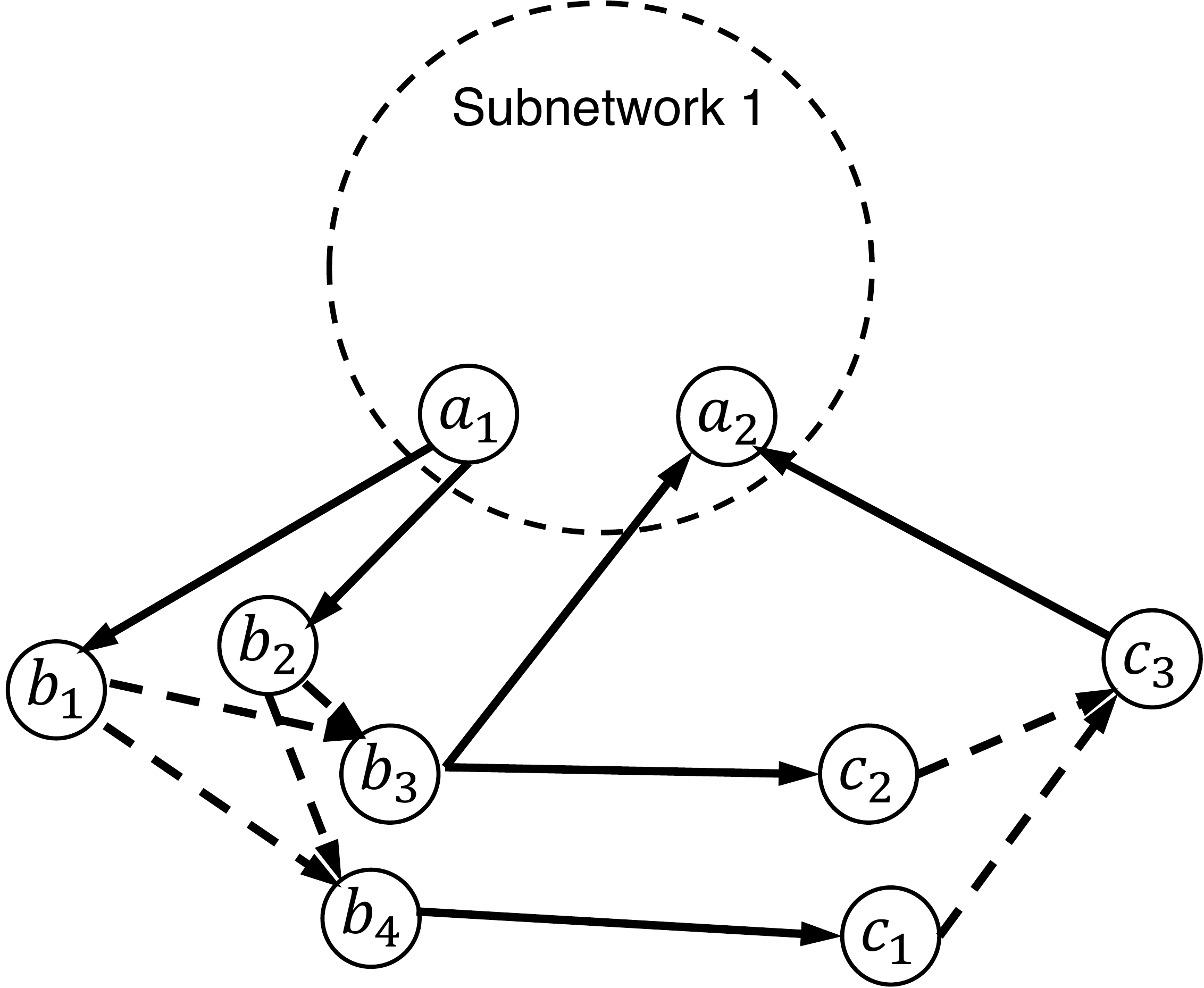}}}
\caption{Issue 2 with subnetwork artificial links where if more than two partitions are considered (part (a)), then capturing interactions outside of a subnetwork requires handling multiple additional links and nodes for every pair of boundary nodes (as shown in (b)).} \label{fig:issue2DSTAP}
\end{figure}

Last, there is an additional computation burden posed by the issue when a boundary node in a subnetwork connects to multiple boundary nodes of another subnetwork. For example, for the network in Figure~\ref{fig:issue3DSTAP} modeling the associated OD pair in subnetwork 2 for the SAL $(a,d)$ in subnetwork 1 requires modeling extra nodes $a$ and $d$ and regional links as part of subnetwork 2, which increases the computational complexity of DSTAP.

\begin{figure}[h]
\centering
\subfigure[]{%
\resizebox*{4.5cm}{!}{\includegraphics{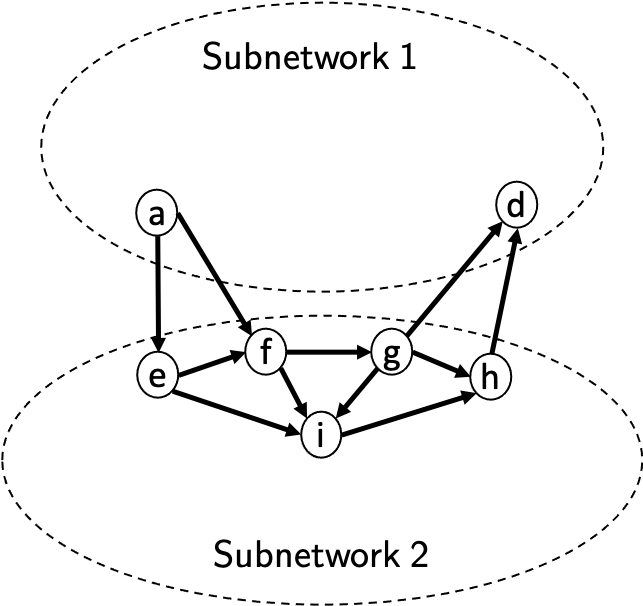}}}\hspace{1pt}
\subfigure[]{%
\resizebox*{4.5cm}{!}{\includegraphics{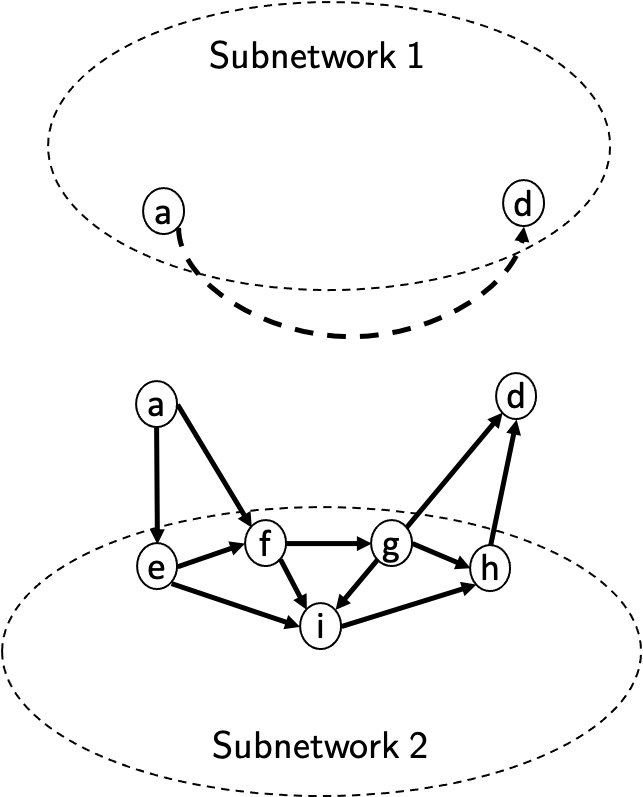}}}
\caption{Issue 3 where modeling the associated OD pair in subnetwork 2 for the SAL $(a,d)$ in subnetwork 1 requires modeling extra nodes $a$ and $d$ and regional links as part of subnetwork 2.} \label{fig:issue3DSTAP}
\end{figure}

Given that the objective of DSTAP is to improve the computational efficiency of solving the TAP  by creating geographically separable partitions, the three issues created due to SALs create computational bottlenecks. In this research, we modify DSTAP by simply dropping the SALs (networks in Figure~\ref{fig:visualDstap}(c)). Artificial links in the master network are left as it is, since they are essential in representing routes connecting interior nodes with other subnetworks. Dropping SALs is a reasonable assumption when there is limited interaction between subnetworks.  Because we exclude SALs, the convergence of DSTAP is no longer guaranteed and DSTAP now becomes a heuristic, which we call DSTAP-Heuristic. Algorithm~\ref{algo:disAlgo} shows the steps for DSTAP-Heuristic in detail, where TAP subproblems are solved using the  Goldstein-Levitin-Polyak gradient projection algorithm~\citep{jaya}.
	
	
	\begin{algorithm}[]
			\caption{DSTAP-Heuristic for solving TAP}
			\label{algo:disAlgo}
			\begin{algorithmic}[]
			\State \textbf{Inputs}: network file, trips file, networks partitions as subnetworks
			\State \textbf{Step 1}: Set up --- validate partitions for connectivity, create master networks and subnetworks, generate master net artificial links and corresponding artificial origin-destination (OD) pairs inside subnetworks. 
			\State \textbf{Step 2}: Optimize
			\While {relative gap on the full network $>$ threshold AND iteration number $<$ desired maximum iterations}
				\State Solve master network to a desired gap
				\State Update the demand on subnetwork artificial OD pairs
				\State Solve the subnetworks in parallel with the updated demand to a desired gap
				\State Update the parameters for master network artificial links
				\State Map flows on used paths between OD pairs in the master network and subnetworks to the full network and compute the full network relative gap
			\EndWhile
			
			\State \textbf{Outputs}: Final full network gap and link flows/costs
			\end{algorithmic}
		\end{algorithm}

Due to the heuristic nature of the algorithm, the relative gap values obtained on the full network after each iteration of DSTAP-Heuristic are not guaranteed to decrease or converge to the optimal solution. However, removing SALs allows each iteration of DSTAP-Heuristic to be faster than DSTAP, which as we demonstrate in Section~\ref{sec:results}, can be used to warmstart the TAP problem and obtain faster solution than using a centralized method.




\subsection{Partitioning algorithms}
Computational performance of DSTAP and DSTAP-Heuristic also depends on the choice of subnetwork boundaries. A good partition is the one that creates fewer boundary nodes and where the subnetworks are ``balanced" in size and have minimal ``interactions" among each other. \cite{yahia2018network} compared two partitioning algorithms: the Shortest Domain Decomposition Algorithm (SDDA), which incrementally builds the subnetworks using seed nodes located far apart, and a flow-weighted Spectral partitioning algorithm that minimizes the flow across all cut arcs (assuming the equilibrium flows are available). 

While both SDDA and Spectral partitioning algorithms can be applied for the DSTAP-Heuristic, we desire partitions that minimize the flows leaving the subnetwork and entering back again, as this minimizes the role of SALs providing better convergence. For a network with two subnetwork partitions, we define a statistic $\psi$ as the difference between the total flow across the cut arcs (interflow) and the total demand for OD pairs with origins and destinations in different subnetworks (interdemand). That is, 
\begin{equation}
    \psi
    = \text{Interflow} - \text{Interdemand}.
    \label{eq:psi}
\end{equation}
If at equilibrium no path between an OD pair leaves a subnetwork and enters it back again (implying zero flow on all SALs), then $\psi=0$. We hypothesize that partitions with lower value of $\psi$ generate solutions with lower relative gap than the ones with higher $\psi$. As discussed in Section~\ref{sec:results}, lower $\psi$ values result in a solution with lower or equal relative gap (after 1 iteration of DSTAP), compared to solutions with higher $\psi$ values.

We propose an iterative partitioning-refinement algorithm which is based on refinement approach proposed in~\cite{fiduccia1982linear} (commonly referred as the FM refinement). Given an initial partition, the algorithm performs local improvement by evaluating the boundary nodes and moving one of the nodes that results in the highest reduction in the current value of $\psi$, until a fixed number of iterations or until no refinement is possible. The details of the $\psi$-FM algorithm are similar to the FM algorithm shown in Algorithm~\ref{algo:psiFM}.  
%

\begin{algorithm}[H]
\begin{algorithmic}
\caption{$\psi-$FM algorithm for improving partitions}
\label{algo:psiFM}
\State Input: given partition of the network with an evaluated $\psi$
\While{$\psi$ of partition cannot be lowered further}
\State Evaluate $\psi$
\State Loop through all boundary nodes and compare the reduction in $\psi$ if nodes were to be shifted from one partition to the next
\State Select the node which leads to highest reduction in the value of $\psi$
\State Revise the partition after shifting the selected node
\EndWhile
\end{algorithmic}
\end{algorithm}

In our experiments, we consider the $\psi$-FM refinement with different initial partitions including SDDA, Spectral, and the partitions generated by the well-known partitioning solver METIS~\citep{karypis1998fast}. We adopt two variants of the METIS solver:
\begin{itemize}
    \item METISv1: this partitioning heuristic partitions a network into $n$ balanced partitions with minimal total flow across all the cut links. It shares the same objective as the Spectral partitioning, though unlike Spectral partitioning it utilizes multi-level partitioning heuristic than the eigenvalue decomposition. See the documentation for flow-weighted METIS solver for additional details~\citep{karypis1998fast}.
    \item METISv2: this partitioning heuristic is same as METISv1, except that the flow weight on each link is replaced with a value of 1. Effectively, METISv2 solves for a partition with minimal number of cut links.
\end{itemize}

We note that the definition of $\psi$ in Equation~\eqref{eq:psi} can be extended to more than one partition; however, the iterative $\psi-$FM algorithm is not guaranteed to converge for more than two partitions~\citep{fiduccia1982linear} . Thereby, we focus our experiments with $\psi-$FM for networks with two-partitions and leave the detailed analysis for more than two partitions for future work.

\section{Other modeling details for DSTAP-Heuristic}
\label{sec:otherDetails}
We now present some specific modeling choices for the DSTAP-Heuristic necessary for full reproducibility of results.

\subsection{Constrained shortest path problem for the master network}
An artificial link in the master network either connects an origin inside each subnetwork with a boundary node of the subnetwork, or connects a boundary node with a destination in the subnetwork. This artificial link models all the paths in the subnetwork connecting its head and tail nodes. However, depending on the choice of partition, it is possible that origin and destination nodes may be located at the boundary of the subnetwork, resulting in feasible paths that include two consecutive artificial links. For example, the master network shown in Figure~\ref{fig:masterNetTransform}(a) (corresponding to the full network in Figure~\ref{fig:visualDstap})  includes both nodes $8$ and $11$ as destinations in subnetwork 2, where node $8$ is also a boundary node (whose incoming links include both an artificial link and a physical link).

We argue that a path in the master network must not include two consecutive artificial links. This is because, if there were two consecutive artificial links $x$ and $y$, then the subpath $[x,y]$ effectively models all paths within a subnetwork connecting the head node of $x$ (denoted by $x_h$) with the tail node of $y$ (denoted by $y_t$); however, these paths should instead be modeled through a direct artificial link connecting $x_h$ and $y_t$. 

Standard label-correcting algorithms cannot be used to find shortest paths under this constraint, as the Bellman property is not satisfied. For example, consider the destination nodes $8$ and $11$ in the master network shown in Figure~\ref{fig:masterNetTransform}(a). The shortest path to node $11$ which does not contain two consecutive artificial links is $\left[ 1,5,7,6,8,11\right]$; however, the subpath $\left[ 1,5,7,6,8\right]$ is not the SP to node $8$. We now present two approaches for finding the shortest paths under the constraint of no consecutive artificial links.
 
\begin{figure}[h]
\noindent\makebox[\textwidth]{%
\includegraphics[scale=0.5]{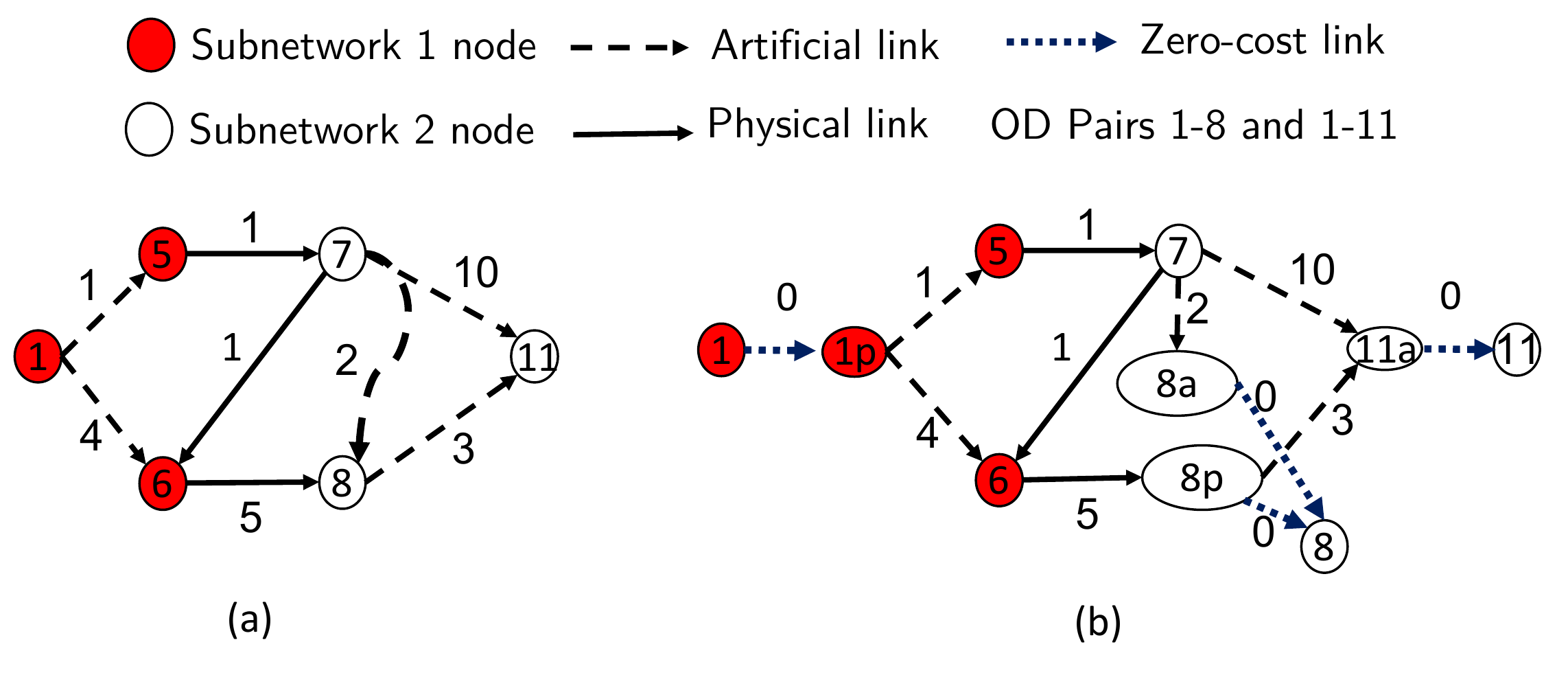}}
\centering
\caption{(a) Original master network, and the (b) transformed master network}
\label{fig:masterNetTransform}
\centering
\end{figure}

\paragraph*{Approach 1 Network transformation for any number of subnetworks} First, we propose a network transformation for solving the constrained SP problem. The proposed transformation creates additional nodes and links for every zone in the network and is presented in detail in Algorithm~\ref{algo:netTransformation}. 

\begin{algorithm}[H]
\begin{algorithmic}
\caption{Master Network Transformation}
\label{algo:netTransformation}
\State Input: Master network $G_{\m} =(N_{\m} ,  A_{\m} , Z_{\m} )$
\State Output: Transformed master network $G_{t} =(N_{t} ,  A_{t} , Z_{t} )$
\vspace{3mm}
\State Set transformed network to be identical to the given master network: $G_t \leftarrow G_\m$
\State Define $N_{\text{child}}$ and $A_{\text{0cost}}$ as empty sets.
\For{all centroid nodes (or zones) $z \in Z_t$}
    \State Create two ``child nodes" $z_p$ and $z_a$ and add it to sets:
    \State \indent $N_t \leftarrow N_t \cup \{z_p,z_a\}, \quad N_{\text{child}} \leftarrow N_{\text{child}} \cup \{z_p,z_a\} $
    \If{$z$ is only an origin node}
    \State Create two links with zero costs: $(z, z_p)$ and $(z, z_a)$ 
    \State Add the created links to sets $A_t$ and $A_{\text{0cost}}$
    \ElsIf{$z$ is only a destination node}
    \State Create two links with zero costs: $(z_p, z)$ and $(z_a, z)$ 
    \State Add the created links to sets $A_t$ and $A_{\text{0cost}}$
    \ElsIf{$z$ is both an origin and a destination node}
        \State Duplicate the zone for its origin and destination equivalent: $z_o$ and $z_d$
        \State Create four links with zero costs: $(z_o, z_p)$, $(z_o, z_a)$,  $(z_p, z_d)$ and $(z_a, z_d)$ 
        \State Add the created links to sets $A_t$ and $A_{\text{0cost}}$
    \EndIf
    
\EndFor
\For{all centroid nodes $z \in Z_t$}
    \For {all incoming links $l\in \Gamma^{-1}_z$ such that $l \notin A_{\text{0cost}}$} 
        \State Remove link $l$ from incoming link set of node $z$: $\Gamma_{z}^{-1} \leftarrow \Gamma_z^{-1} \setminus \{ l \}$
        \If{$l$ is a physical link}
            \State Add $l$ to incoming set of child node $z_p$: $\Gamma_{z_p}^{-1} \leftarrow \Gamma_{z_p}^{-1} \cup \{ l \}$
        \ElsIf{$l$ is an artificial link}
            \State Add $l$ to incoming set of child node $z_a$: $\Gamma_{z_a}^{-1} \leftarrow \Gamma_{z_a}^{-1} \cup \{ l \}$
        \EndIf
    \EndFor
\EndFor
\For{all centroid nodes $z \in Z_t$}
    \For {all outgoing links $l\in \Gamma_z$ such that $l \notin A_{\text{0cost}}$} 
        \State Remove link $l$ from outgoing link set of node $z$: $\Gamma_{z} \leftarrow \Gamma_z \setminus \{ l \}$
        \If{$l$ is a physical link}
            \State Add copy of link $l$ as an outgoing link for each child node $z_p$ and $z_a$:
            \State $\qquad \Gamma_{z_p} \leftarrow \Gamma_{z_p} \cup \{ l \}, \quad \Gamma_{z_a} \leftarrow \Gamma_{z_a} \cup \{ l \}$
                
        \ElsIf {$l$ is an artificial link}
            \State Add copy of link $l$ as an outgoing link for the corresponding $z_p$ node:
            \State $\qquad \Gamma_{z_p} \leftarrow \Gamma_{z_p} \cup \{ l \}$
        \EndIf
    \EndFor
\EndFor
\State Remove all child nodes without any incoming or outgoing link $l \in A_t \setminus A_{\text{0ccost}}$.
\end{algorithmic}
\end{algorithm}

Figure \ref{fig:masterNetTransform}(b) shows how node $8$ is split into child nodes $8_a$ and $8_p$. Nodes with subscript $a$ can only have artificial link as incoming links and only physical link as outgoing links. On the other hand, nodes with subscript $p$ can only have physical link as incoming links and allow any link type as outgoing links. Hence, node $8_a$ does not include the original outgoing link $(8,11)$ since both its incoming and outgoing links cannot be artificial.  Next, we show that using standard label-correcting algorithms on the transformed network solves the constrained SP problem. 

\begin{lemma}
    All paths from an origin to a destination in the transformed network obtained from Algorithm~\ref{algo:netTransformation} satisfy the no-consecutive-artificial-link path constraint
    \label{lem:part1}
\end{lemma}
\begin{proof}
    The proof follows from the design of transformed network where in its last step the algorithm does not add outgoing artificial links to a child node ($z_a$) that has an incoming artificial link.
\end{proof}


\begin{lemma}
    There exists a path in transformed network for every feasible path in the original network
    \label{lem:part2}
\end{lemma}
\begin{proof}
    The proof is included in Appendix~\ref{appendix:proof}.
\end{proof}

From Lemma~\ref{lem:part1} we can argue that a shortest path found in transformed network is also a feasible path in the original network. Similarly, from Lemma~\ref{lem:part2} we can argue that a shortest path found in transformed network must be the shortest feasible path in the original network, because if there were any other shorter feasible path, then the same path would exist in the transformed network and would be found. Hence, solving the shortest path problem on the transformed network will provide the shortest path in the original master network that does not include two consecutive artificial links. 

The network transformation approach in Algorithm~\ref{algo:netTransformation} is general and can be applied to partitions with more than two subnetworks. Unfortunately, the proposed  transformation creates several additional nodes and links in the master network. In our experiments on large-scale networks, we observe that this increased network size can slow the computation of the shortest path by up to twice the amount relative to when path constraints are removed, which can be detrimental to the desired computational efficiency of DSTAP-Heuristic.  

\paragraph*{Approach 2: For two subnetworks only} The second approach designed for two-subnetwork partition addresses the computational-efficiency issue while preserving path constraints. This approach exploits the property that a shortest path for master network with two subnetworks can have at most 3 links, and simply adapts the one-to-all label-correcting algorithms for finding the shortest paths in master network by avoiding two consecutive artificial links. Assuming the origin is in subnetwork 1 and destinations are in subnetwork 2, the algorithm completes the search of a shortest path in the master network in three stages as follows:

\begin{enumerate}
    \item In stage 1 all outgoing artificial links from the origin are scanned and the cost labels of link's tail nodes are updated. Because an artificial link exists from origin to every boundary node in subnetwork 1, this stage updates the cost label of all boundary nodes in subnetwork 1. We note that if the origin is located at the boundary, this step does not scan the physical links that might connect the origin to boundary nodes of subnetwork 2.
    
    \item In stage 2, the cost labels of all boundary nodes are updated, considering only the connections formed using physical links. An iterative label-setting procedure is used that starts with all boundary nodes in subnetwork 1 as part of the scan-eligible list (SEL). In each iteration, the node with lowest cost label is removed from SEL and all its outgoing physical links are scanned to add other potential nodes to SEL. Since all boundary nodes in subnetwork 2 are connected to a boundary node in subnetwork 1 through a physical link, this stage guarantees that cost labels of all boundary nodes will be updated. We note that it is possible for the label of a boundary node in subnetwork 1 to be modified in this stage if there exists an alternate shortest path using the physical links from another boundary node.
    
    \item Stage 3 enumerates paths to all destinations in subnetwork 2 by exploiting the property that path constraint violation can only occur in this stage. We enumerate the shortest path for each destination in subnetwork 2 as back node labels are not sufficient to construct the shortest path since Bellman's principle does not hold. Due to the property that the master net shortest path can have at most 3 links, storing explicit shortest paths does not add significant memory constraints.  In first part of this stage, paths to all interior destinations are enumerated by scanning all incoming artificial links to the destination. In second stage, the same process is repeated for all destinations at the boundary; however, now the destinations are scanned in increasing order of their cost labels. Destination nodes at boundary already have cost labels from stage 2 and this order of scanning ensures that we find a path that may include another incoming artificial link to the node. 
\end{enumerate}

Algorithm~\ref{algo:3StageDijkstra} provides the pseudocode for the 3-stage shortest path algorithm (termed \textit{3-stage-SPP}) on master network. For the network in Figure~\ref{fig:masterNetTransform}(a), the three stages proceed as follows: in stage 1, cost labels of nodes 1, 5, 6 are updated to 0,1, and 4, respectively; in stage 2, cost labels of nodes 5,6,7, and 8 are updated to 1,3,2, and 8 respectively, and in stage 3, paths to destination 11 is updated as [1,5,7,11] with a cost of 12 units and to destination 8 is updated as [1,5,7,8] with a cost of 4 units.

It is easy to verify that 3-stage-SPP algorithm indeed finds the optimal path with desired property of no-consecutive artificial links for two subnetwork partitions since it  accounts for artificial links only in first and third stages which explicitly captures the path constraint rule of no consecutive artificial links. Extending the algorithm for more than 2 partitions will involve more than three stages and is left as part of the future work.

\begin{algorithm}
\begin{algorithmic}[1]
\caption{One-to-all 3-stage-SPP algorithm for solving constrained shortest path problem on master network with two partitions}
\label{algo:3StageDijkstra}
\State Input: master network with physical and artificial links, associated link costs (given by LinkCost($i,j$) function for a link $(i,j)$), and an origin node $o$ (assumed to be in subnetwork 1)
\vspace{3mm}
\State \textbf{Initialization}: For the origin node, set cost label $c_o=0$. For all other nodes $i\in N_{\text{master}}\setminus \{o\}$, set $c_i=\infty$ and back node label $b_i = \text{null}$.
\vspace{3mm}
\State \textbf{Stage 1}: Update shortest cost labels for boundary nodes in subnetwork 1
\For{all outgoing artificial links $(o,i) \in \Gamma_o$}
    \If{$c_o + \text{LinkCost}(o,i) < c_i$}
        \State Set $c_i \leftarrow c_o + \text{LinkCost}(o,i)$ and $b_i \leftarrow o$
    \EndIf
\EndFor
\vspace{3mm}
\State \textbf{Stage 2}: Update shortest cost labels for all boundary nodes
    \State Initialize scan eligible list ($SEL$) to all boundary nodes in subnetwork 1
    \While{$SEL$ is not empty}
        \State Remove a node from $SEL$ with lowest cost label. Call it $n$.
        \For{all outgoing physical links $(n,i) \in \Gamma_n$}
            \If{$c_n + \text{LinkCost}(n,i) < c_i$}
                \State Set $c_i \leftarrow c_n + \text{LinkCost}(n,i)$ and $b_i \leftarrow n$
                \State $SEL \leftarrow SEL \cup \{ i\}$
            \EndIf
        \EndFor
    \EndWhile
\vspace{3mm}
\State \textbf{Stage 3}: Define shortest path to each destination in subnetwork 2
    \State Define empty shortest path variable  $\pi_d \leftarrow \text{null}$ for all destinations $d$ in subnetwork 2.
    \vspace{1mm}
    \State \textit{Stage 3a}: Update paths to interior destinations
    \For{all destination $d$ in subnetwork 2 that are not boundary nodes}
    \For{all incoming artificial links $(i,d) \in \Gamma_d^{-1}$}
    \If{$c_i + \text{LinkCost}(i,d) < c_d$}
        \State Set $c_d \leftarrow c_i + \text{LinkCost}(i,d)$
        \State Append node $d$ to path to $i$: $\pi_d \leftarrow \texttt{tracePathTo}(i) \oplus  d$
    \EndIf
    \EndFor
    \EndFor
    \vspace{1mm}
    \State \textit{Stage 3b}: Update paths to destinations that are at the boundary
    \For{all destinations $d$ in subnetwork 2 that are boundary nodes in increasing order of current cost labels}
    \For{all incoming artificial links $(i,d) \in \Gamma_d^{-1}$}
    \If{$c_i + \text{LinkCost}(i,d) < c_d$}
        \State Set $c_d \leftarrow c_i + \text{LinkCost}(i,d)$
        \State Append node $d$ to path to $i$: $\pi_d \leftarrow \texttt{tracePathTo}(i) \oplus  d$
    \EndIf
    \EndFor
    \EndFor
\end{algorithmic}
\end{algorithm}

Next, we determine the computational complexity of the 3-stage-SPP algorithm. If $b_i$, $o_i$, and $d_i$ are respectively the number of boundary nodes, origins, and destinations in subnetwork $i$, $i\in \{1,2\}$, the computational complexity of stage 1 is $\mathcal{O}(b_1)$, stage 2 is $\mathcal{O}\left( (b_1+b_2)^2 \right)$, and Stage 3 is $\mathcal{O}(d_2b_2 + b_2^2)$. Since megaregional networks have significantly higher number of destinations in a subnetwork than the number of boundary nodes, the algorithm has an overall complexity of $\mathcal{O}\left( d_2 b_2\right)$. When repeated for all origins, the all-to-all shortest path routine for master network has the complexity $\mathcal{O}\left( o_1d_2 b_2 + o_2 d_1 b_1\right)$, which is bounded by $\mathcal{O}\left( |Z_m|^2|B|\right)$, where $Z_m$ is the set of zones in master network and $B$ is the set of all boundary nodes.


We also note that if the path constraints are not modeled, the procedure for DSTAP-Heuristic in Algorithm~\ref{algo:disAlgo} still holds, albeit it might result in a further suboptimal solution (since the proof in \cite{jafari2017decomposition} relies on the assumption that regional OD pair paths have alternating artificial and physical links.) If we apply standard one-to-all Dijkstra's algorithm to solve shortest paths in the master network, the overall computational complexity is $\mathcal{O}\left( |Z_m||N_m|^2\right)$, where $N_m$ is the set of all nodes in the master network. Since $|B|<<|N_m|$ and $|Z_m|<|N_m|$, the 3-stage-SPP algorithm is more efficient and saves additional computation time. In Section~\ref{sec:results} we quantify this savings for various networks. 


\subsection{Handling centroid connectors and network pruning}
Transportation networks consist of centroids (or zones) where trips originate and end. Commonly, it is assumed that first $\zeta$ nodes in the network are centroids and any link originating or terminating at those nodes is a centroid connector. If a partitioning algorithm finds a cut passing through a centroid connector, it can result in a disconnected OD pair which is undesirable. We circumvent this problem by first removing all centroids and centroid connectors from the network prior to the partitioning, then applying a partitioning algorithm on the modified network, and then reattaching the centroids to the appropriate nodes and partition. If during the reattachment process, a centroid is found connected to nodes in different subnetworks, copies of the centroid are created one for each subnetwork, to ensure the connectivity of OD pairs in the master network and subnetworks. Furthermore, to model the constraint that the centroids should only be at the ends of any path in a network, we set the first through node for the masternetwork and subnetworks to be the same value as the complete network, which is equal to $\zeta+1$ by the convention in \cite{tntp}.

In addition to potential duplication of centroids, we note that for partitioning using Spectral and METISv1 algorithms, we drop the links that have zero flow. This is consistent with the design of METISv1~\citep{karypis1998fast} and Spectral~\citep{bell2017investigating,yahia2018network} algorithms to prevent finding a cut that passes through zero flow links as this often results in an imbalanced cut.  For transportation networks, zero flow links are a result of unused portions of network that are not used by a traveler. Dropping the zero flow links might create disconnected portions of network. In our analysis, we consider the largest connected portion for further partitioning and drop the other portion. This dropping of nodes and links can thus result in DSTAP-Heuristic working on a smaller portion of the network than the original network; however, there is no loss of accuracy as the TAP can be solved separately on other smaller disconnected portions.

\subsection{Best possible gap achieved using DSTAP-Heuristic}
The performance of DSTAP-Heuristic depends on the choice of the partition and whether there is ``significant" proportion of demand leaving a subnetwork and entering back again. 

As stated earlier, the value of $\psi-$statistic plays a role in the quality of partition. If we are able to find a partition with $\psi=0$, then DSTAP-heuristic (as described in Section~\ref{subsec:dstapvsheuristic}) converges to the optimal within a threshold as argued in the original DSTAP work. Unfortunately, if we relax path constraints in the master network or have a partition where $\psi>0$, then convergence is not guaranteed. 

 In our experiments, we found that DSTAP-Heuristic shows initial convergence, though the relative gap begins to oscillate as the iterations are continued. Given its heuristic nature, we report the best possible relative gap for DSTAP-Heuristic obtained across all iterations. The value of best possible relative gap depends on (a) the number of cut links/boundary nodes, and (b) the value of $\psi$. In the next section, we demonstrate that despite its heuristic nature and lack of theoretical guarantees, DSTAP-Heuristic is a useful algorithm for improving computational efficiencies of solving the TAP.
 

\section{Results}
\label{sec:results}

We conduct experiments on six mid- to large-scale test networks obtained from the transportation networks test problems repository~\citep{tntp} including Berlin Center, Austin, Goldcoast, Philadelphia, Chicago Regional, and Texas\footnote{The Texas network was obtained from the statewide regional model and is the largest network among all instances with 44798 nodes, 122944 links, and 4667 zones. Other networks statistics are provided in the Appendix~\ref{appendix:partition_stats}.}. All experiments are conducted on a Linux machine Ubuntu 18.04, 32 GB of memory and Intel i5 processor \@ 3.30 GHz. 

First, we report the benefits of using DSTAP-Heuristic over DSTAP.  Table \ref{tab:dstapVSheuristic} shows the number of additional nodes and links generated by DSTAP and DSTAP-Heuristic procedures for different networks partitioned using a given partitioning heuristic.\footnote{As explained in Section~\ref{sec:warmstarting} all partitioning heuristics performed relatively similarly across the test networks. In this section, we report results for the heuristic that took the least time in solving the network to a relative gap of 1E$-4$ after warmstarting DSTAP using the solution obtained from one run of DSTAP-Heuristic.} For each subnetwork and the master network, we report the total number of nodes, number of physical links (ones that are present in the complete network before partitioning), and the number of artificial links generated to model interactions between interior and boundary nodes and to model subnetwork interactions. As observed, modeling subnetwork interaction in DSTAP leads to creation of  0.8\%--2.4\% additional nodes and  1.9\%--21.4\% additional SALs. Each additional SAL has an associated OD pair in the other subnetwork, and thus leads to creation of up to 3.5\% additional origins requiring additional shortest path computation. 
By dropping the SALs, the network in DSTAP-Heuristic is simplified, allowing for computational efficiency over the centralized algorithm that we discuss next. The master networks in both DSTAP and DSTAP-Heuristic are identical with up to 1650 times more artificial links than physical links. However, as discussed in Section~\ref{subsec:dstapvsheuristic}, master network artificial links are essential for the DSTAP-Heuristic process and with an appropriate choice of the shortest path routine, these links do not add significant burden to computations. 


\begin{table}[h]
\centering
\caption{Comparison of number of SALs generated by DSTAP and DSTAP-Heuristic. In DSTAP-Heuristic the number of artificial links are simply dropped. The three-tuple values indicate number of nodes, number of physical links, and number of artificial links.}
\label{tab:dstapVSheuristic}
\begin{tabular}{|c|c|c|}
\hline
\textbf{Network} & \textbf{DSTAP} & \textbf{DSTAP-Heuristic} \\ \hline
Austin & \begin{tabular}[c]{@{}c@{}}Masternet- (1094, 81, 74079)\\ Subnetwork 1- (3777, 9242, 1441)\\ Subnetwork 2- (3867, 9376, 1510)\end{tabular} & \begin{tabular}[c]{@{}c@{}}Masternet- (1094, 81, 74079)\\ Subnetwork 1- (3685, 9242, 0)\\ Subnetwork 2- (3779, 9376, 0)\end{tabular} \\ \hline
 Berlin & \begin{tabular}[c]{@{}c@{}}Masternet- (637, 113, 55416)\\ Subnetwork 1- (6554, 13527, 2466)\\ Subnetwork 2- (6681, 14730, 2413)\end{tabular} & \begin{tabular}[c]{@{}c@{}}Masternet- (637, 113, 55416)\\ Subnetwork 1- (6432, 13527, 0)\\ Subnetwork 2- (6549, 14730, 0)\end{tabular}\\ \hline
 Goldcoast & \begin{tabular}[c]{@{}c@{}}Masternet- (1089, 20, 20756)\\ Subnetwork 1- (2218, 5108, 100)\\ Subnetwork 2- (2609, 6008, 112)\end{tabular} & \begin{tabular}[c]{@{}c@{}}Masternet- (1089, 20, 20756)\\ Subnetwork 1- (2194, 5108, 0)\\ Subnetwork 2- (2587, 6008, 0)\end{tabular}\\ \hline
 Chicago & \begin{tabular}[c]{@{}c@{}}Masternet- (1919, 159, 262868)\\ Subnetwork 1- (6509, 19229, 4673)\\ Subnetwork 2- (6776, 19630, 5351)\end{tabular} & \begin{tabular}[c]{@{}c@{}}Masternet- (1919, 159, 262868)\\ Subnetwork 1- (6351, 19229, 0)\\ Subnetwork 2- (6628, 19630, 0)\end{tabular}\\ \hline
 Philadelphia & \begin{tabular}[c]{@{}c@{}}Masternet- (1623, 158, 223258)\\ Subnetwork 1- (6874, 20606, 5811)\\ Subnetwork 2- (6819, 19239, 4950)\end{tabular} & \begin{tabular}[c]{@{}c@{}}Masternet- (1623, 158, 223258)\\ Subnetwork 1- (6728, 20606, 0)\\ Subnetwork 2- (6661, 19239, 0)\end{tabular}\\ \hline
Texas & \begin{tabular}[c]{@{}c@{}}Masternet- (1729, 192, 207593)\\ Subnetwork 1- (22467, 61780, 7998)\\ Subnetwork 2- (22599, 60520, 7640)\end{tabular} & \begin{tabular}[c]{@{}c@{}}Masternet- (1729, 192, 207593)\\ Subnetwork 1- (22293, 61780, 0)\\ Subnetwork 2- (22421, 60520, 0)\end{tabular} \\ \hline
\end{tabular}
\end{table}

Next, we demonstrate the computational advantage of using the 3-stage-SPP algorithm relative to solving Dijkstra's algorithm on the entire Master network. Figure~\ref{fig:3-stage-SPP-comparison} shows the time it took to solve master network using the 3-stage-SPP algorithm and solving Dijkstra's on the entire master network. 
As observed, the 3-stage-SPP algorithm is up to 10 times more efficient than solving Dijkstra's algorithm on the entire network because it exploits the layout of master network where artificial links only appear in first and third stages.

\begin{figure}[H]
\centering
\includegraphics[scale=0.6]{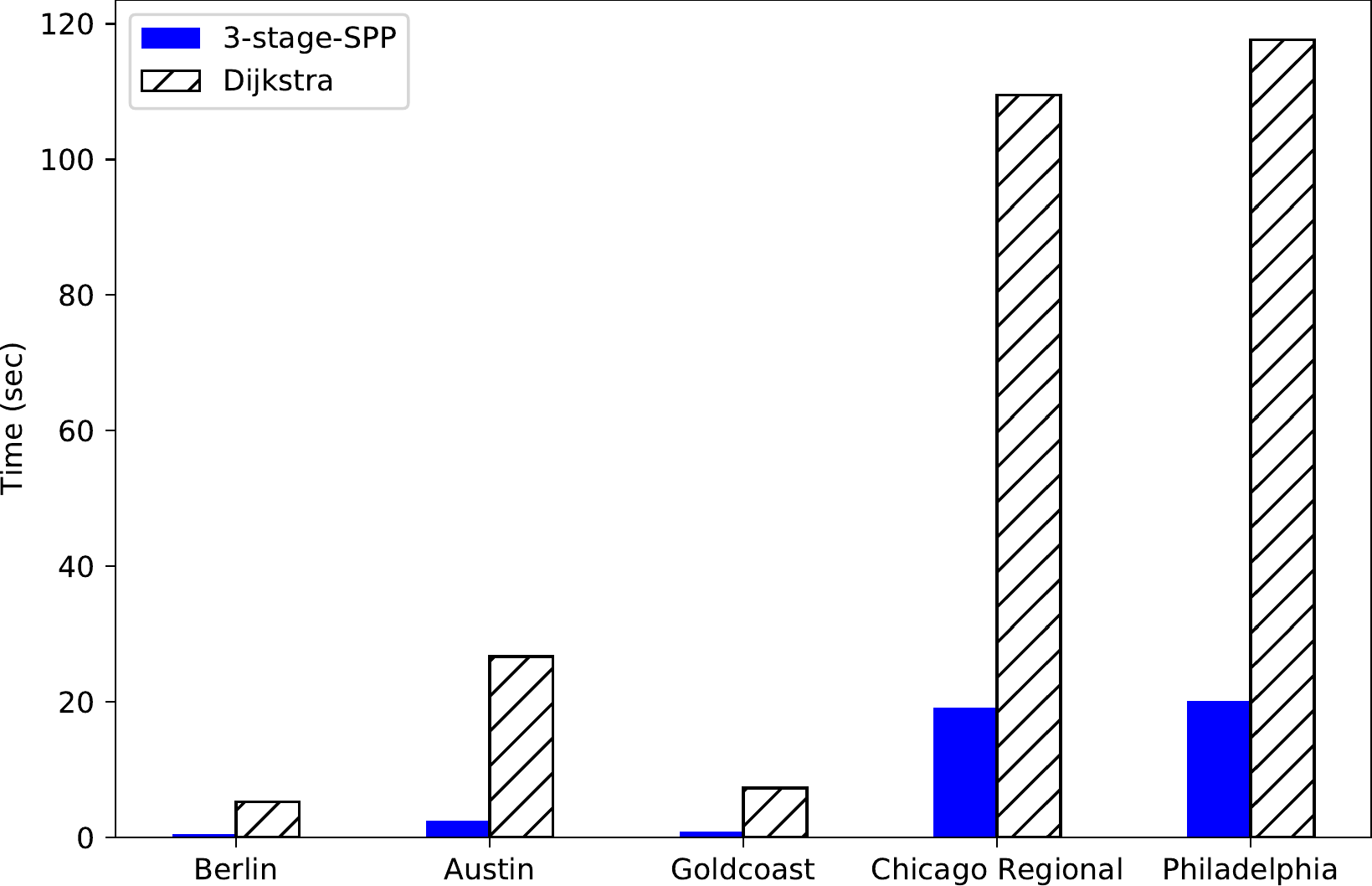}
\caption{Comparison of computation time to solve master network using two different shortest path routines: the 3-stage-SPP algorithm and standard Dijkstra's algorithm} 
\label{fig:3-stage-SPP-comparison}
\end{figure}

Next, we compare the computation times obtained from DSTAP-Heuristic relative to the centralized approach for different test networks, as shown in Table~\ref{tab:allCompTime}. Column 2 shows the lowest relative gap obtained from running DSTAP-Heuristic and the partition obtaining that gap. Columns 3 and 4 show the computation time taken by DSTAP-Heuristic and the centralized algorithm to reach that gap (using the gradient projection algorithm as the TAP solver for all cases).

\begin{table}[H]
\centering
\caption{Comparison of computation time for generating solutions up to a certain relative gap using DSTAP-Heuristic vs gradient projection (GP)}
\label{tab:allCompTime}
\begin{tabular}{|c|c|c|c|c|}
\hline
\multicolumn{2}{|c|}{} & \multicolumn{3}{c|}{\textbf{Computation time till best gap (sec)}} \\ \hline
\textbf{Network} & \textbf{\begin{tabular}[c]{@{}c@{}}Best gap\\ DSTAP-Heuristic\end{tabular}} & \textbf{\begin{tabular}[c]{@{}c@{}}DSTAP-Heuristic\\ using GP\end{tabular}} & \textbf{\begin{tabular}[c]{@{}c@{}}Centralized\\ using GP\end{tabular}} & \textbf{\%-savings} \\ \hline
\textbf{Berlin} & 0.0273 (METISv2) & \textbf{18.21} & 41.88 & 56.5\% \\ \hline
\textbf{Austin} & 0.1712 ($\psi$-FM METISv2) & \textbf{48.61} & 86.09 & 43.5\% \\ \hline
\textbf{Goldcoast} & 0.0608 (Spectral) & \textbf{47.94} & 100.92 & 52.5\% \\ \hline
\textbf{Chicago} & 0.0744 (METISv2) & \textbf{231.87} & 439.76 & 47.3\% \\ \hline
\textbf{Philadelphia} & 0.7226 ($\psi$-FM METISv2) & \textbf{241.09} & 284.11 & 15.1\% \\ \hline
\textbf{Texas} & 0.0175 (METISv2) & \textbf{1464.43} & 4551.87 & 67.8\% \\ \hline
\end{tabular}
\end{table}

We make following key observations:
\begin{itemize}
	\item First, the best-possible achieved gap using the DSTAP-Heuristic depends on the choice of partitioning algorithm and varies in the range of 0.02--0.72 for various networks. Due to its heuristic nature, DSTAP-Heuristic is not guaranteed to monotonically reduce the relative gap from one iteration to the next.
	\item We observe that, with the exception of Goldcoast network, METISv2 partitions or its variant obtained using $\psi-$FM refinement did the best in obtaining the lowest gap after an iteration of DSTAP-Heuristic relative to other partitioning algorithms. METISv2 simply relies on the indegree and outdegree of network nodes and generates a balanced partition relative to the ones generated by other algorithms. METISv2 is also advantageous over Spectral and METISv1 partitions since it does not rely on prior knowledge of equilibrium flows that are sometimes unavailable. We also observe that the $\psi$-FM refinement does better than the partitions without this refinement for 2 out of the 6 tested networks.
	\item For all network instances, DSTAP-Heuristic took lesser time than the centralized approach (solved to same relative gap as the best possible achieved gap by DSTAP-Heuristic) generating a percent time savings ranging between 15.1--67.8\%. This demonstrates the usefulness of DSTAP-Heuristic to obtain efficient approximation to TAP solutions than using a centralized approach.
\end{itemize}

Figure~\ref{fig:splitDSTAPTime} shows the split of computation times across four subroutines of DSTAP-Heuristic. The split is reported as a percent of the total time spent in (a) solving the master network to a prespecified relative gap of 0.05, (b) solving subnetworks in parallel to a relative gap of 0.05, (c) mapping the flow obtained on DSTAP-Heuristic network to the full network, and (d) computing full network relative gap. 
As expected, solving subnetworks in parallel takes a significant portion of the DSTAP-Heuristic time ($>55\%$ for all networks). Master network computations are relatively lower ($<15\%$), credited to the 3-stage-SPP algorithm. Computing the relative gap on the full network takes the second-highest time for all networks since solving full network gap requires solving shortest path on the full network, which is time-consuming. If full network relative gap computation is skipped (such as when the termination criterion is based on performing a fixed number of iterations), DSTAP-Heuristic will prove additionally useful. 

\begin{figure}[h]
\centering
\includegraphics[scale=0.5]{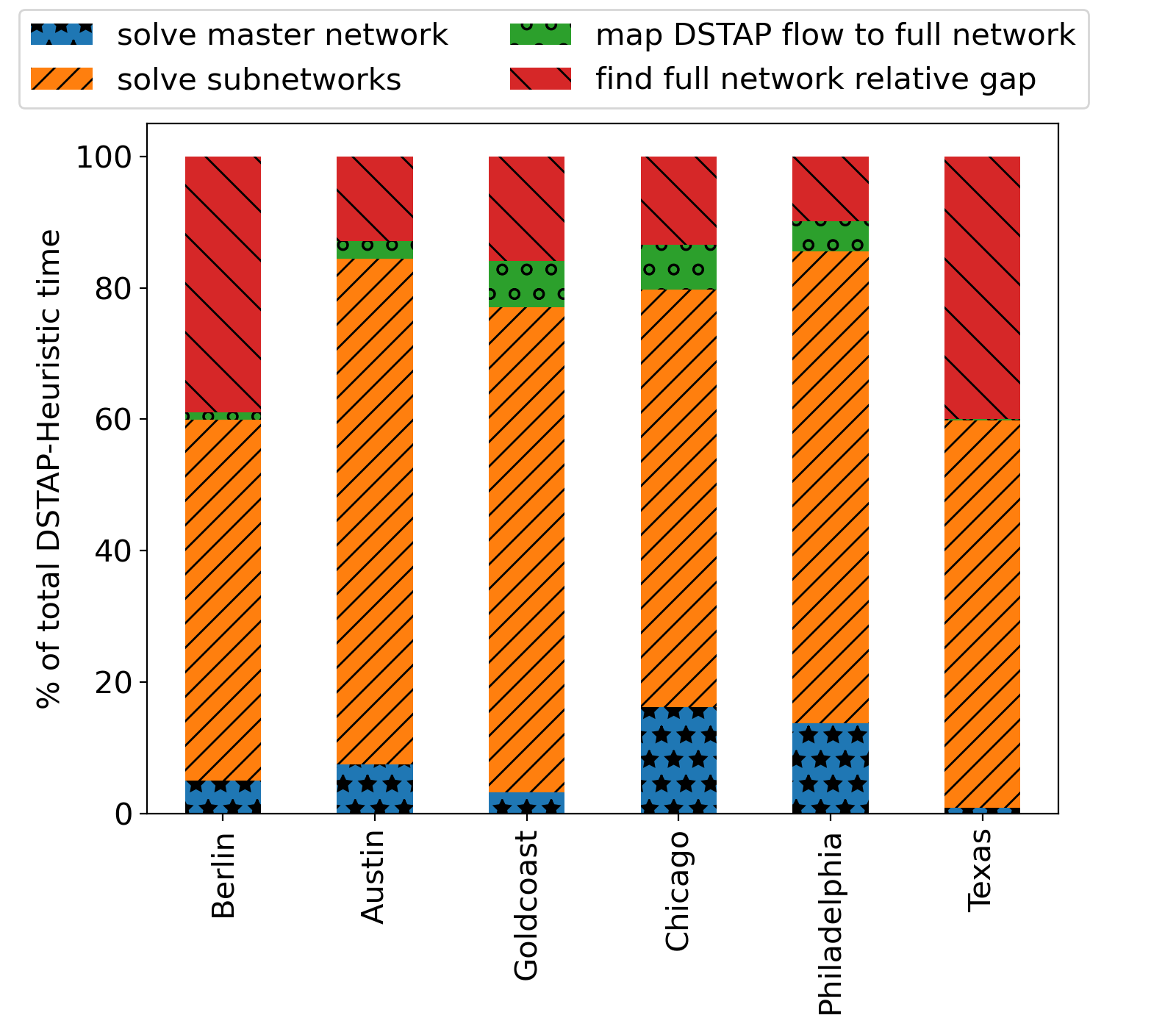}
\caption{Percent of total DSTAP Heuristic from Table~\ref{tab:allCompTime} that is spent in different subroutines} 
\label{fig:splitDSTAPTime}
\end{figure}

\begin{figure}[H]
\centering
\subfigure[Austin $\psi$-FM METISv2 partition]{%
\resizebox*{4.5cm}{!}{\includegraphics{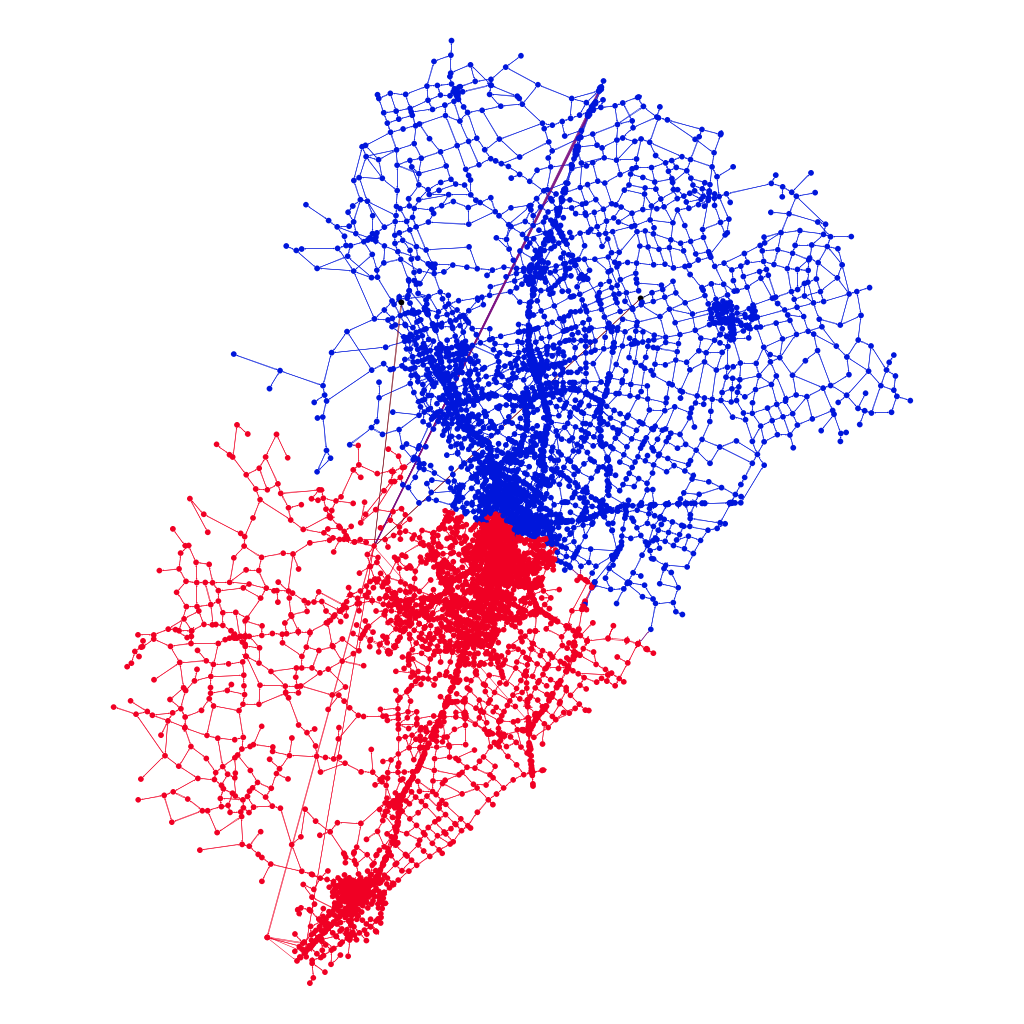}}}\hspace{1pt}
\subfigure[Berlin METISv2 partition]{%
\resizebox*{4.5cm}{!}{\includegraphics{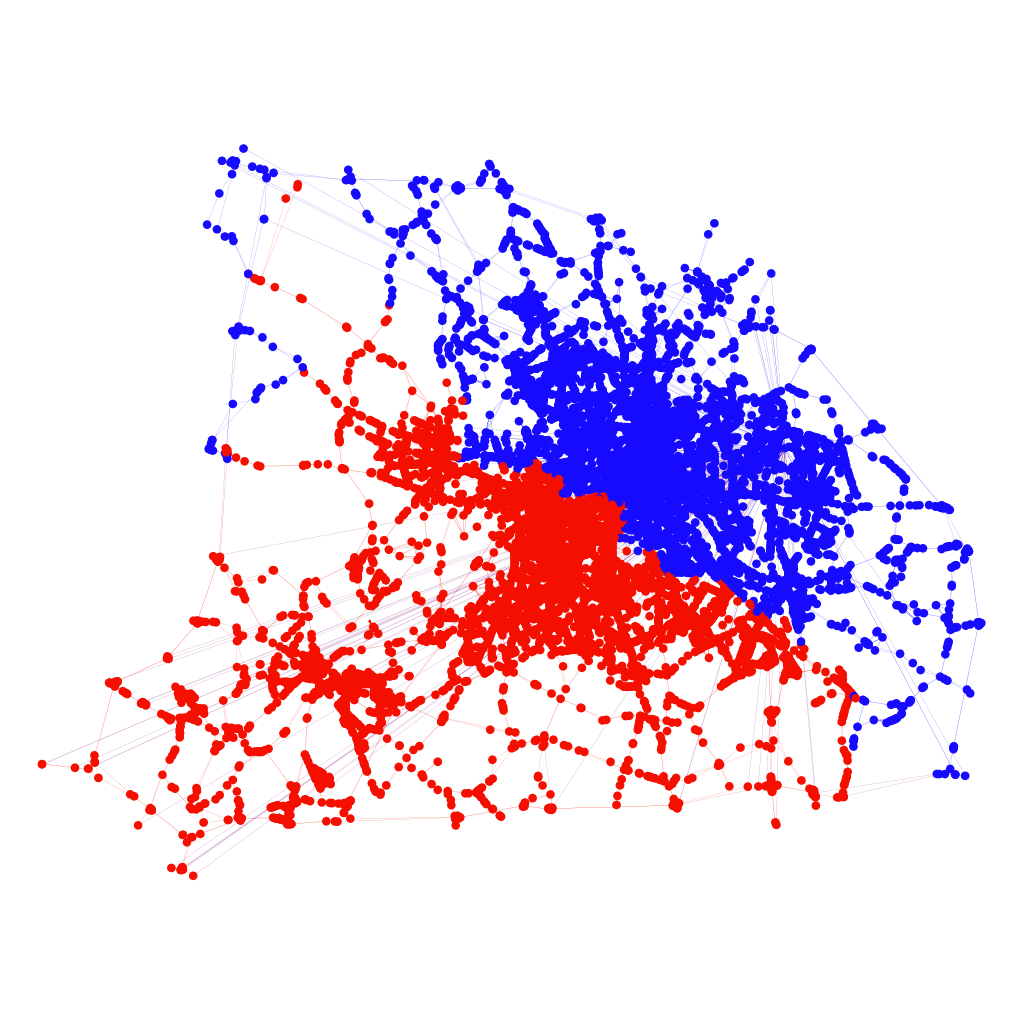}}}\hspace{1pt}
\subfigure[Goldcoast Spectral partition]{%
\resizebox*{4.5cm}{!}{\includegraphics{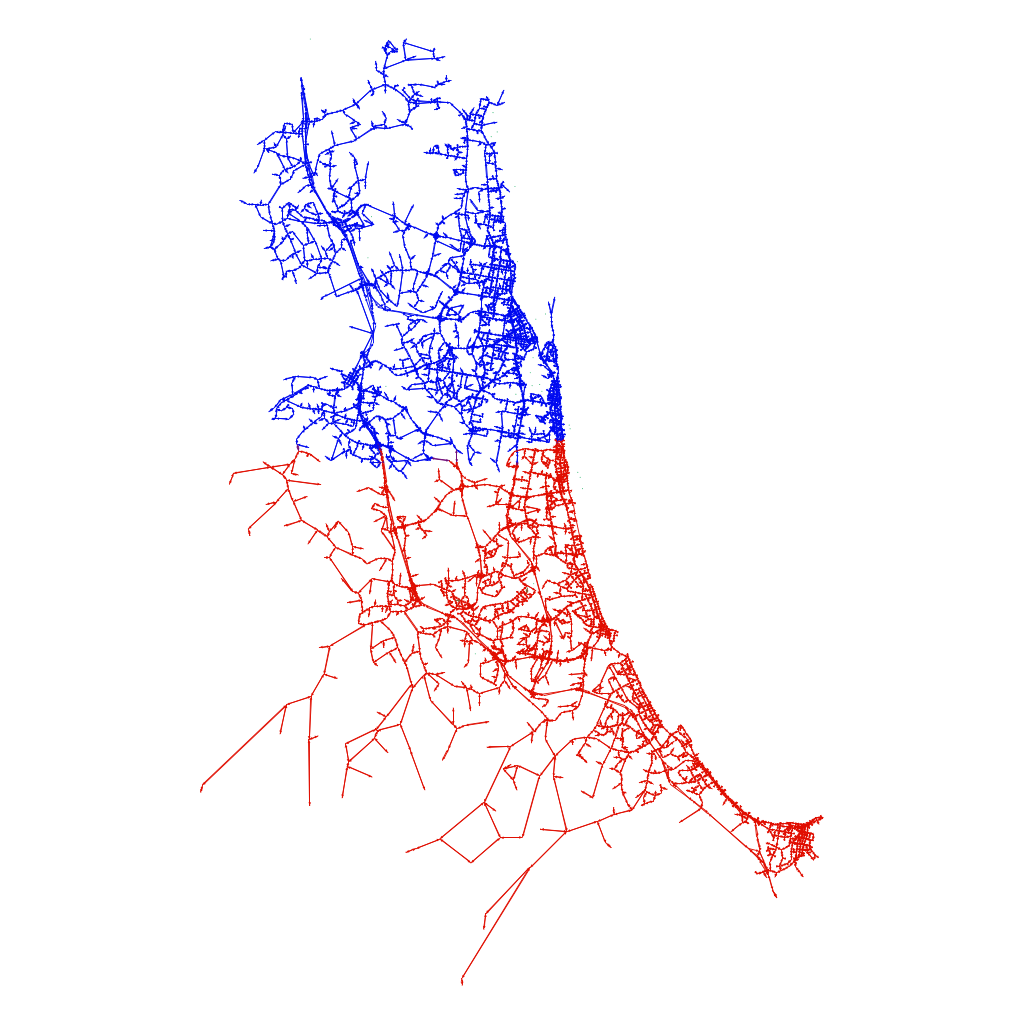}}}\hspace{1pt}
\subfigure[Chicago METISv2 partition]{%
\resizebox*{4.5cm}{!}{\includegraphics{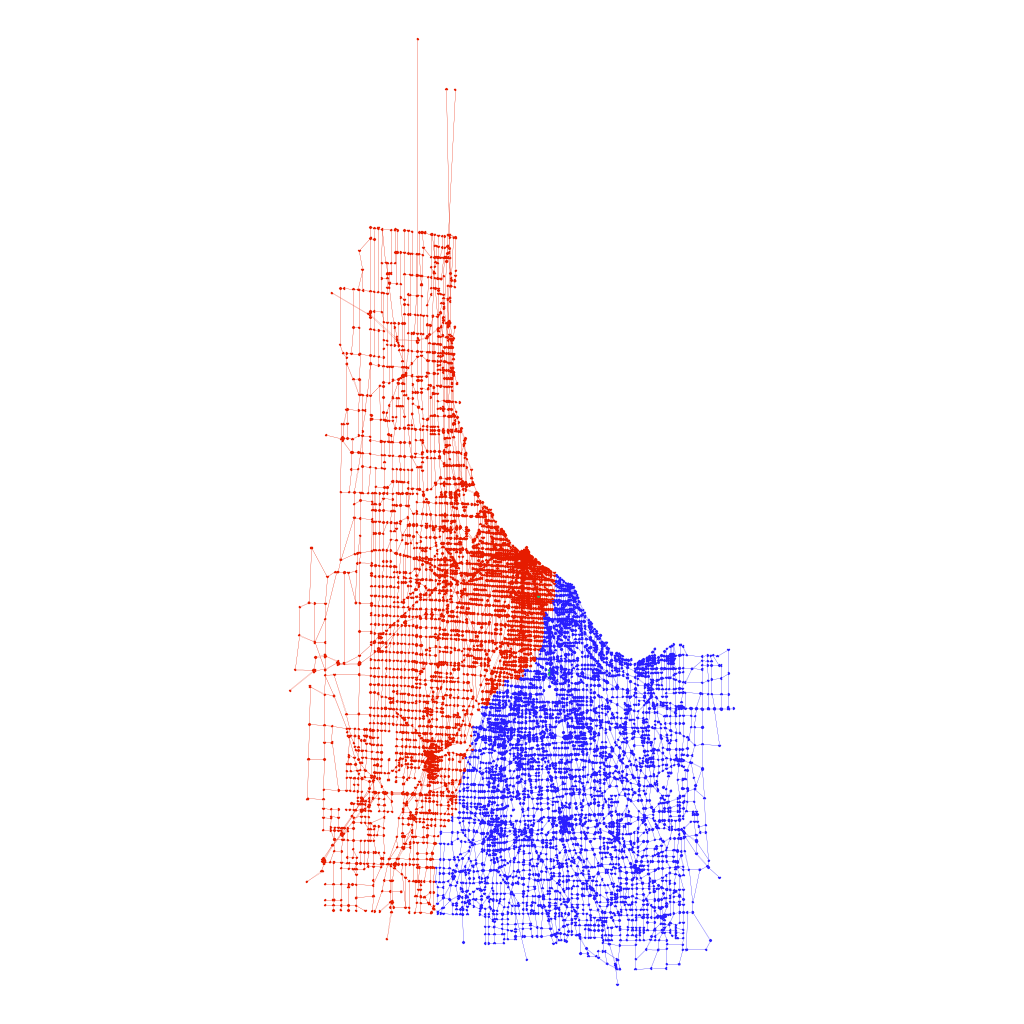}}}\hspace{1pt}
\subfigure[Philadelphia $\psi$-FM METISv2 partition]{%
\resizebox*{4.5cm}{!}{\includegraphics{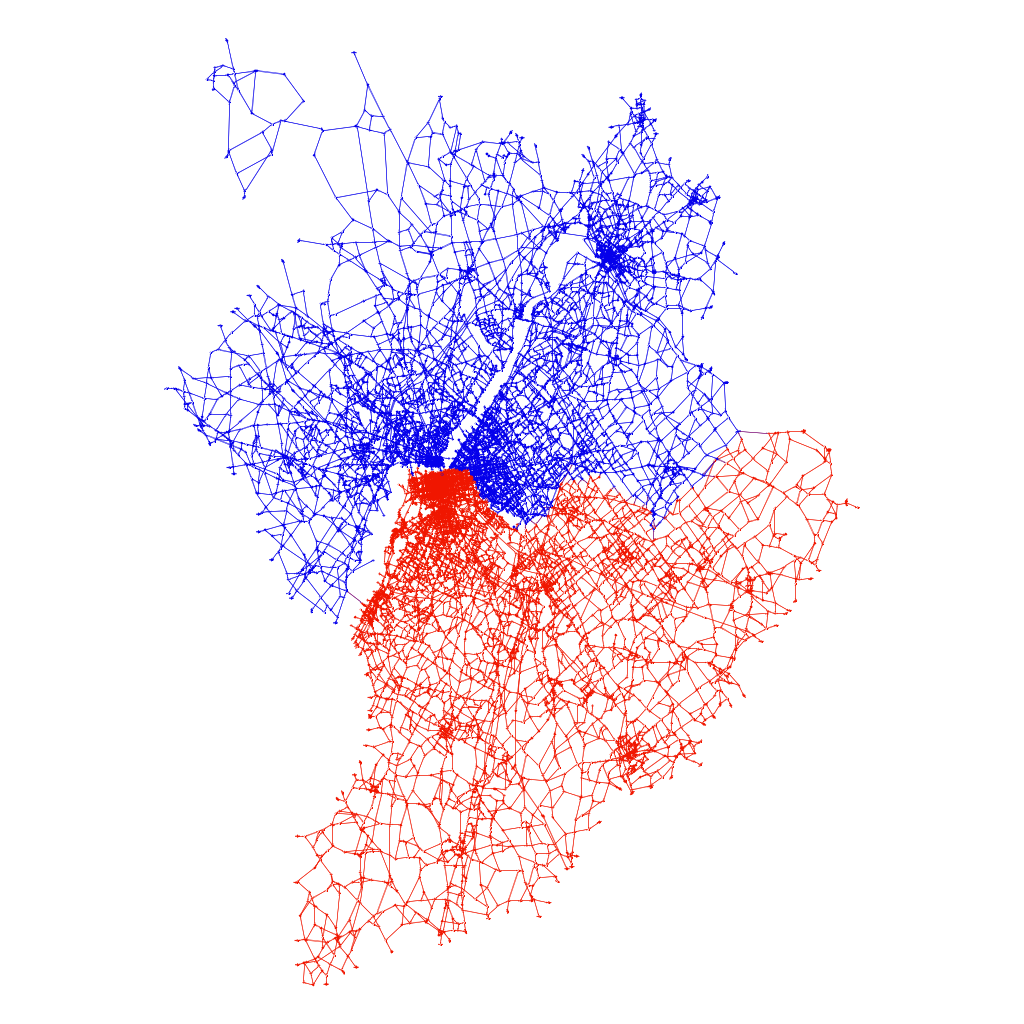}}}
 \subfigure[Texas METISv2 2-net partition]{%
 \resizebox*{4.5cm}{!}{\includegraphics{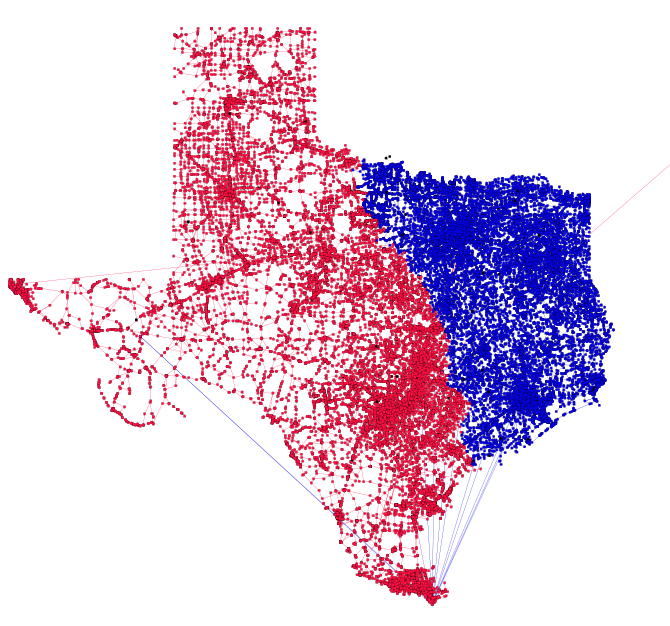}}}\hspace{1pt}
\caption{All network partitions that performed the best relative to the centralized gradient projection} \label{fig:allpart}
\end{figure}

Figure \ref{fig:allpart} shows the partitions on these networks. For all the network instances, the partitions generated are well-balanced and may explore the geographical features of the city. For example, the METISv2 partition for Austin network cuts bridges on the Colorado river for a portion of the partitioning line. Similarly, the Texas network is partitioned with each subnetwork containing the major cities (Dallas and Houston in blue partition, and Austin, San Antonio and El Paso in the red partition). For the Philadelphia network, while the Delaware river runs through the city, dividing the network just using the bridges as cut links would have resulted in imbalanced partition. Instead, $\psi$-FM METISv2 partition uses a portion of river as the boundary between subnetworks and the remainder of the boundary passes through the city. The partitions for Berlin, Goldcoast and Chicago Regional networks are standard balanced partitions without a clear geographical interpretation. While the best achieved relative gap is still high using these partitions is high, the computational advantage of DSTAP-Heuristic makes it suitable for obtaining an approximate TAP solution or for warmstarting a TAP using standard algorithms which we demonstrate next.


\subsection{Warm Starting TAP using DSTAP-Heuristic Solution}
\label{sec:warmstarting}
Due to its heuristic nature, DSTAP-Heuristic is not guaranteed to converge. In this section, we report improvements obtained from using DSTAP-Heuristic to warm start a traditional TAP algorithm (in our case, gradient projection). This warm starting guarantees convergence to the optimal solution and is faster than the traditional TAP algorithm for all tested networks. 

Figure~\ref{fig:all_net_warmsart_results} presents results for the set of experiments where warm-starting was implemented after one iteration of DSTAP-heuristic for the networks mentioned earlier.

First, as a general trend, we observe that  the proposed heuristic performs better than centralized gradient projection for all tested networks. However, there are partitions that take longer than the centralized algorithm. METISv2 did consistently better for all networks. 
Second, depending on the network characteristics, the performance of DSTAP-Heuristic may only be marginally better. The observed benefit narrows significantly for the Gold Coast and Philadelphia networks, where warm-starting using the heuristic performs marginally better than the centralized algorithm. 

\begin{figure}
\centering
\subfigure[Austin]{%
\resizebox*{7cm}{!}{\includegraphics{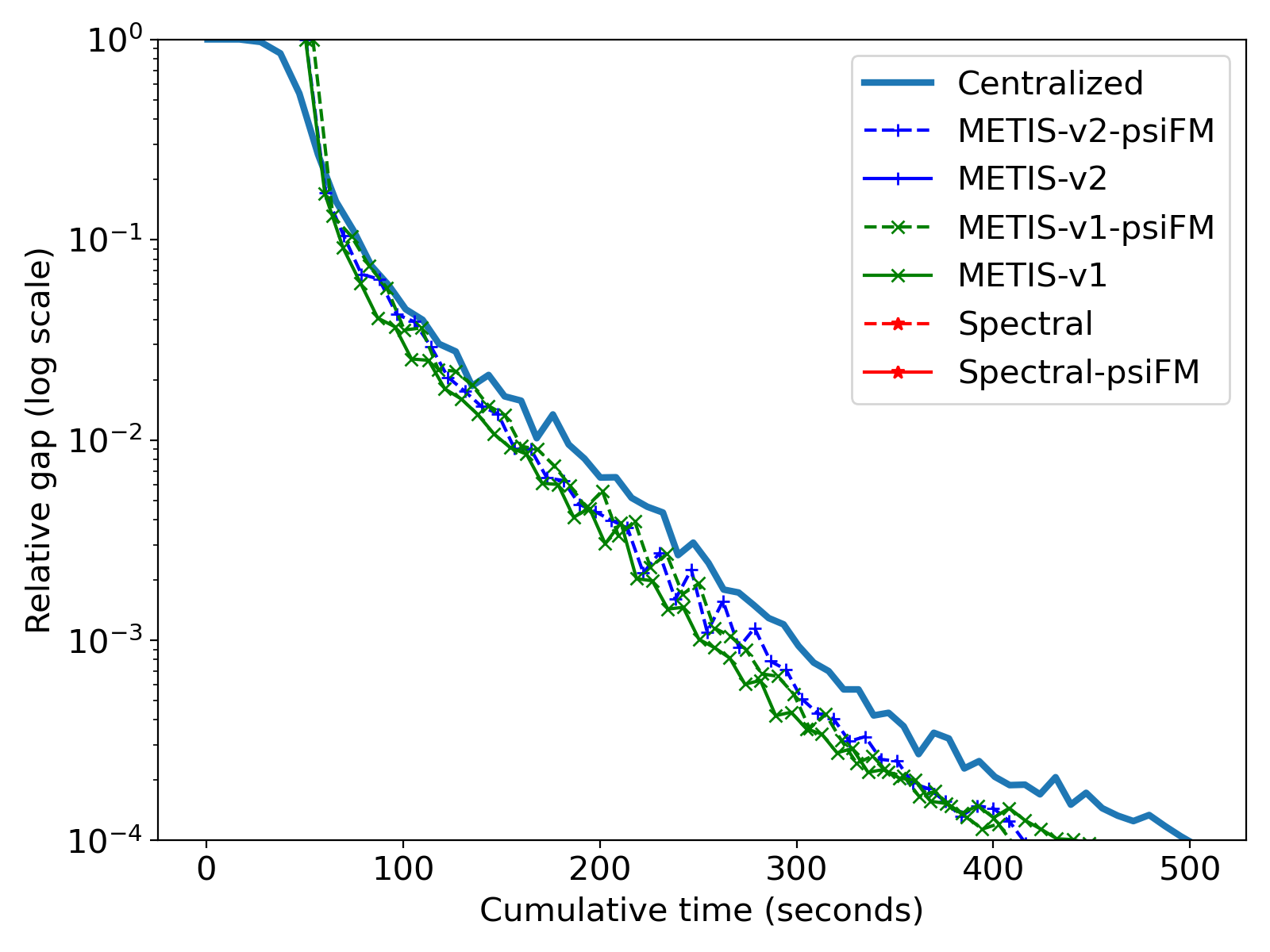}}}\hspace{1pt}
\subfigure[Berlin]{%
\resizebox*{7cm}{!}{\includegraphics{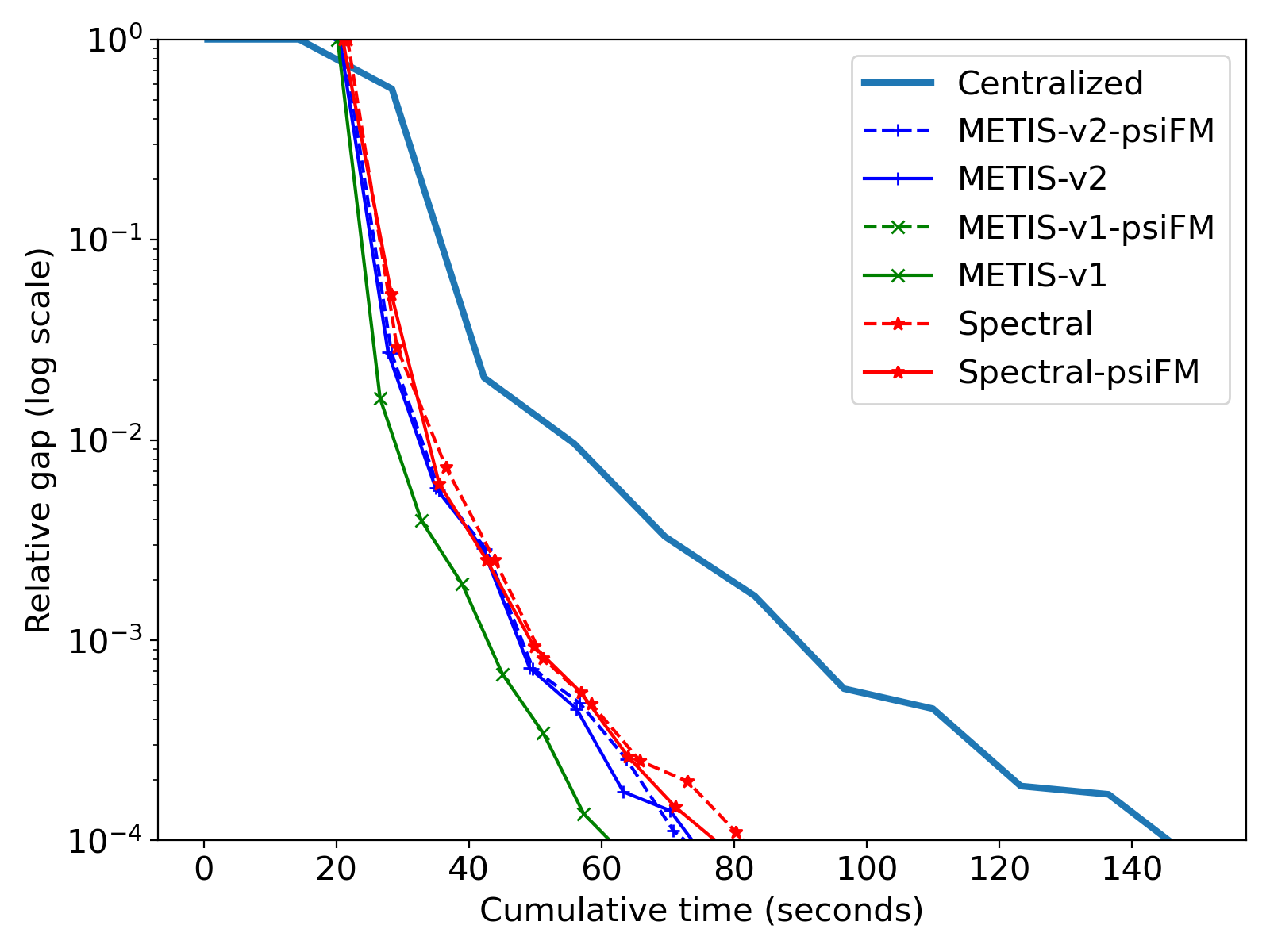}}}\hspace{1pt}
\subfigure[Goldcoast]{%
\resizebox*{7cm}{!}{\includegraphics{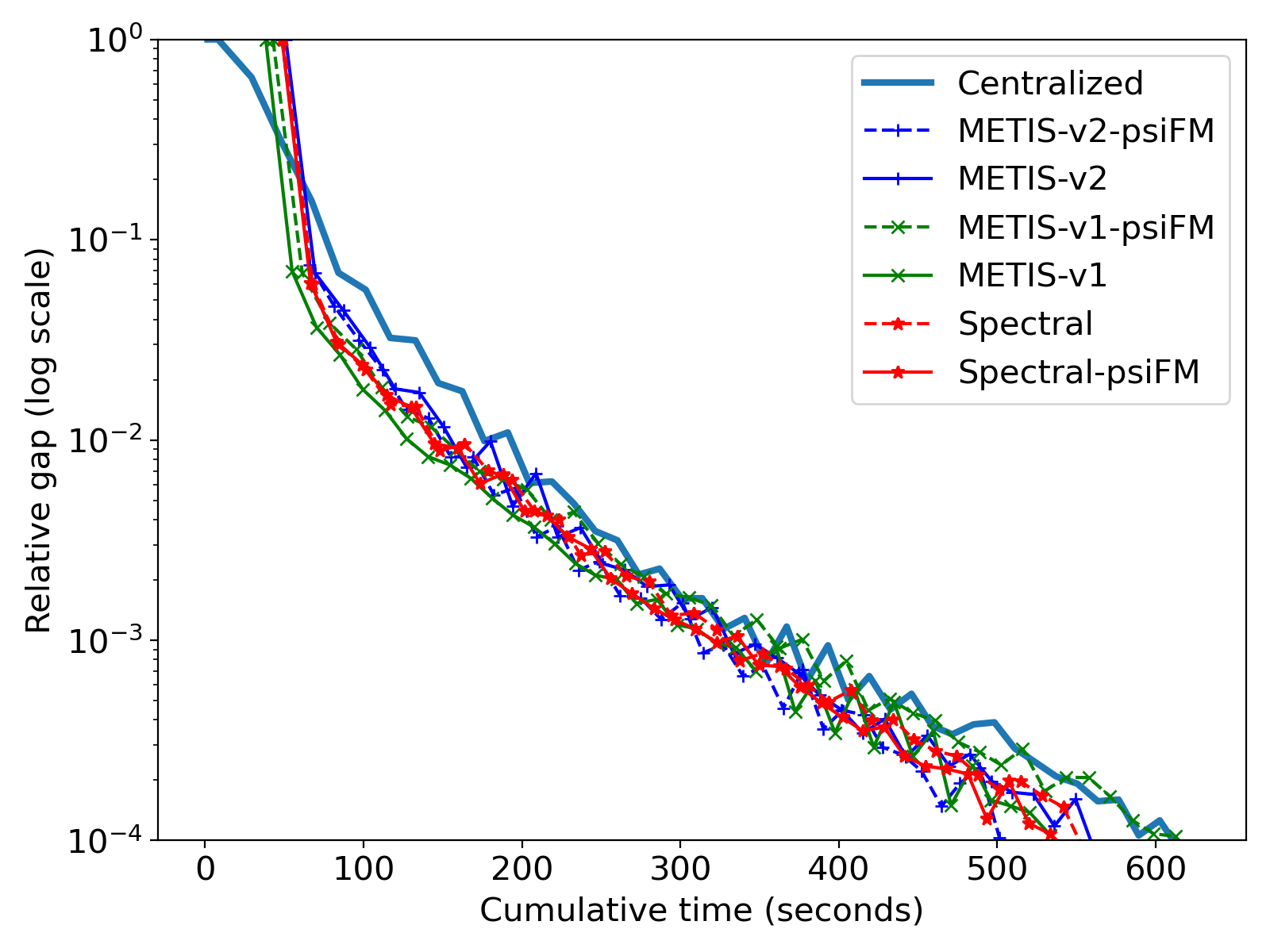}}}
\subfigure[Chicago]{%
\resizebox*{7cm}{!}{\includegraphics{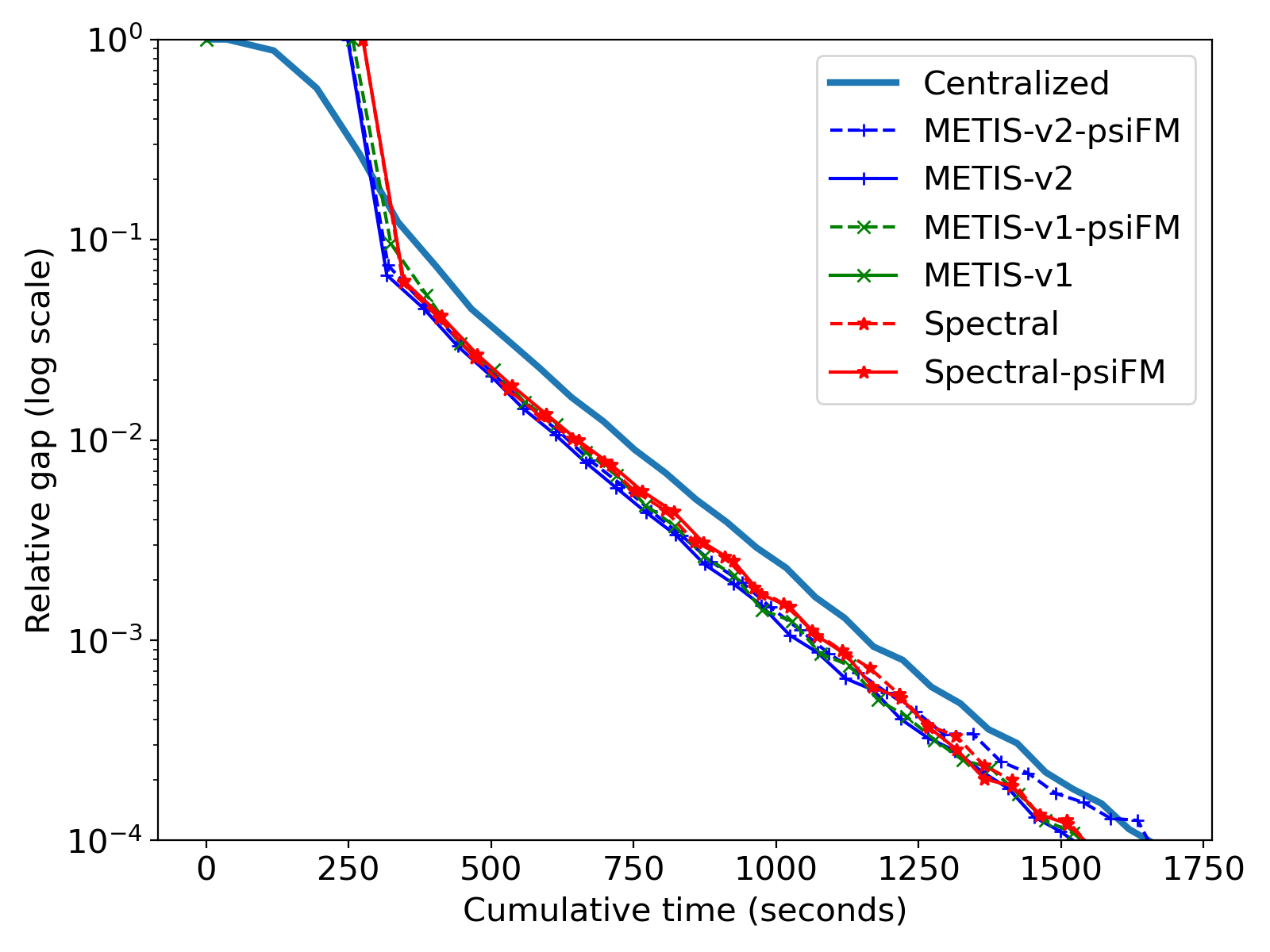}}}
\subfigure[Philadelphia]{%
\resizebox*{7cm}{!}{\includegraphics{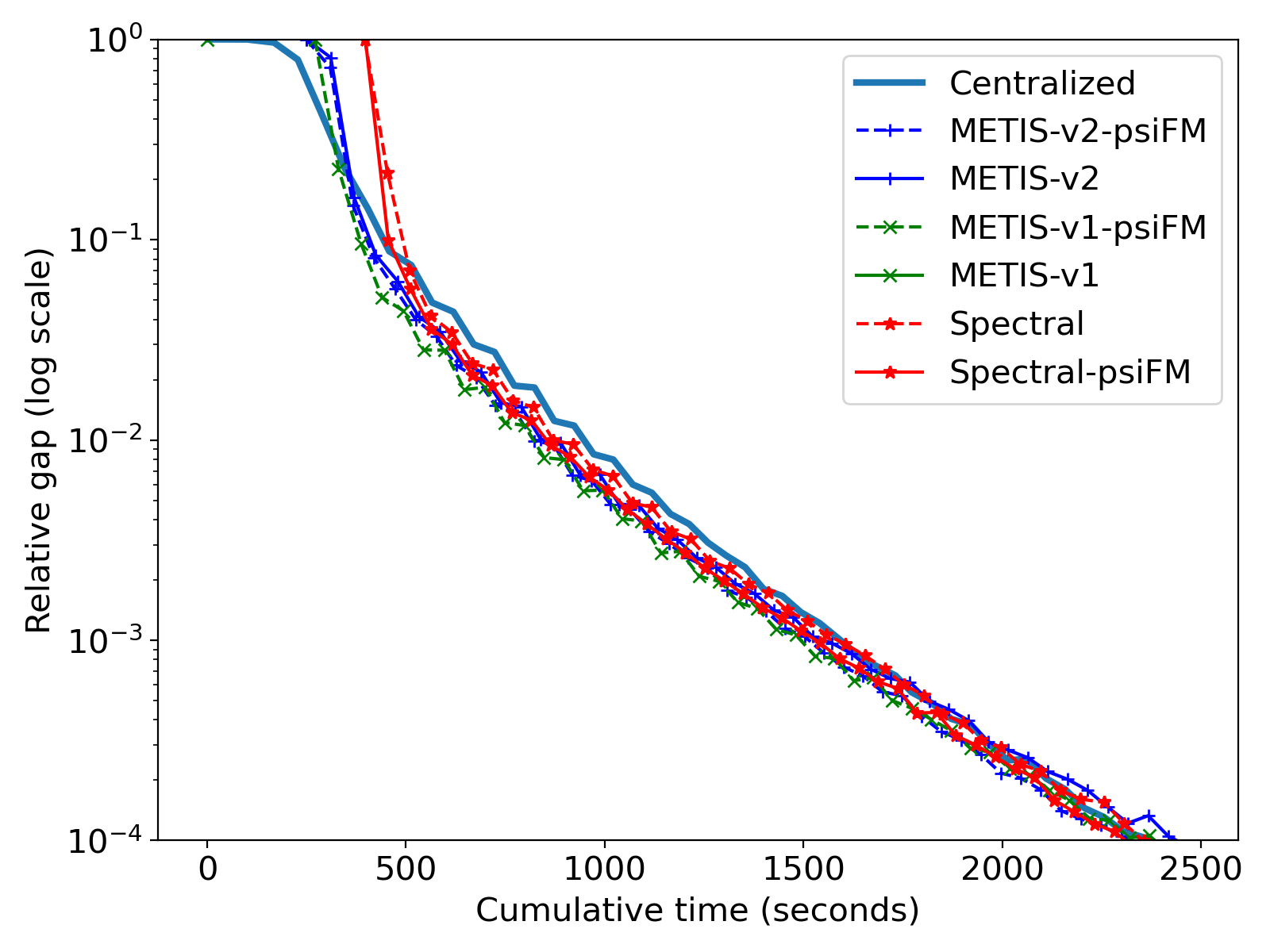}}}
\caption{Plots of relative gap as a function of computation time after warmstarting gradient projection using the path flows obtained after 1 iteration of DSTAP-Heuristic for various networks.} 
\label{fig:all_net_warmsart_results}
\end{figure}




Table~\ref{tab:demandVariants} presents the percent savings obtained from using DSTAP-Heuristic for warmstarting TAP for the best partition instance for three demand levels with scaling factors as 1.0 (base demand), 0.85 (low demand), and 1.5 (high demand). We do not generate new partitions for low and high demand scenarios to demonstrate the usefulness of partitions generated at base demand (demand factor 1.0). As observed, for all demand levels and for all networks, DSTAP-Heuristic generates computational savings between 2.0\%--47.05\% relative to centralized gradient projection. There is no clear evidence that increasing the demand increases the benefit of DSTAP-Heuristic; however, this may be due to the choice of same partition as base demand.

The application of proposed DSTAP-Heuristic as a warmstarting tool is highlighted by the results. Beyond the first iteration, the gap for warmstart and no-warmstart case reduce at similar rates due to usage of the same centralized algorithm. Therefore, the cause for the savings is isolated at warmstarting.  
The results provide evidence for usability of the proposed heuristic on large networks and the associated time saving, while overcoming the heuristic gap issue. The detailed numerical results and partition statistics can be found in Appendix \ref{appendix:partition_stats}. 

\begin{sidewaystable}
\centering
\caption{Comparison of computation time for centralized gradient projection when solved to a relative gap of 1E$-4$ without (``Centralized") and with (``DSTAP-Heuristic") warmstarting using the DSTAP-Heuristic algorithm}
\label{tab:demandVariants}
\begin{tabular}{ccccccc}
\hline
\multirow{2}{*}{\textbf{\begin{tabular}[c]{@{}c@{}}Demand\\ Factor\end{tabular}}} & \multirow{2}{*}{\textbf{\begin{tabular}[c]{@{}c@{}}Computation\\ time\end{tabular}}} & \multicolumn{5}{c}{\textbf{Network}} \\ \cline{3-7} 
 &  & \textbf{Austin} & \textbf{\begin{tabular}[c]{@{}c@{}}Berlin-\\ Center\end{tabular}} & \textbf{Goldcoast} & \textbf{\begin{tabular}[c]{@{}c@{}}Chicago\\ Regional\end{tabular}} & \textbf{Philadelphia} \\ \hline
\multirow{3}{*}{\textbf{0.85}} & \textbf{\begin{tabular}[c]{@{}c@{}}Centralized\\ time  (sec)\end{tabular}} & 442.64 & 119.88 & 409.73 & 1530.40 & 2883.65 \\ \cline{2-7} 
 & \textbf{\begin{tabular}[c]{@{}c@{}}DSTAP-Heuristic\\ time (sec)\end{tabular}} & 404.50 & 63.48 & 350.38 & 1165.89 & 2198.49 \\ \cline{2-7}  
 & \textbf{\% savings} & 8.62\% & 47.05\% & 14.49\% & 23.82\% & 23.76\% \\ \hline
\multirow{3}{*}{\textbf{1.0}} & \textbf{\begin{tabular}[c]{@{}c@{}}Centralized\\ time (sec)\end{tabular}} & 684.10 & 144.18 & 616.25 & 1912.71 & 2673.76 \\ \cline{2-7} 
 & \textbf{\begin{tabular}[c]{@{}c@{}}DSTAP-Heuristic\\ time (sec)\end{tabular}} & 425.55 & 78.29 & 555.08 & 1646.68 & 2357.19 \\ \cline{2-7} 
 & \textbf{\% savings} & 37.8\% & 45.70\% & 9.93\% & 13.91\% & 11.84\% \\ \hline
\multirow{3}{*}{\textbf{1.5}} & \textbf{\begin{tabular}[c]{@{}c@{}}Centralized\\ time (sec)\end{tabular}} & 869.42 & 182.9 & 1129.91 & 4752.40 & 4004.96 \\ \cline{2-7} 
 & \textbf{\begin{tabular}[c]{@{}c@{}}DSTAP-Heuristic\\ time (sec)\end{tabular}} & 687.49 & 102.09 & 1162.42 & 3649.99 & 3435.46 \\ \cline{2-7} 
 & \textbf{\% savings} & 20.93\% & 44.18\% & $-2.88$\% & 23.20\% & 14.22\% \\ \hline
\end{tabular}
\end{sidewaystable}





\subsection{Using DSTAP Heuristic for Megaregions}
Last, to demonstrate the usefulness of DSTAP-Heuristic for solving TAP on megaregions, we analyze the performance of the heuristic on the Texas  network which consists of major cities such as Austin, Dallas, Houston, and San Antonio. 

Megaregional networks commonly have higher density of links and nodes in the  city area, and are relatively sparse in other locations. An intuitive way to partition a megaregion would be to consider each city as its own subnetwork; this intuitive partitioning aligns with how the regional models are maintained, where the state handles solving TAP on a relative sparse statewide network, while the MPOs handle solving TAPs on a detailed citywide network. However, it is a difficult process to determine where the boundary of partitions should pass, and which cities should be grouped together, which is where the default graph partitioning algorithms can be used. Due to the superior performance of METISv2 partitions, we use the same for Texas megaregional network. Figure~\ref{fig:allpartTexas} shows the partitions generated using METISv2 algorithm for the Texas megaregion considering two and three subnetworks. 

\begin{figure}[H]
\centering
\subfigure[Texas METISv2 2-net partition]{%
\resizebox*{4.5cm}{!}{\includegraphics{Temp_plots/Texas_METISv2_partition_2subnet.png}}}\hspace{1pt}
\subfigure[Texas METISv2 3-net partition]{%
\resizebox*{4.5cm}{!}{\includegraphics{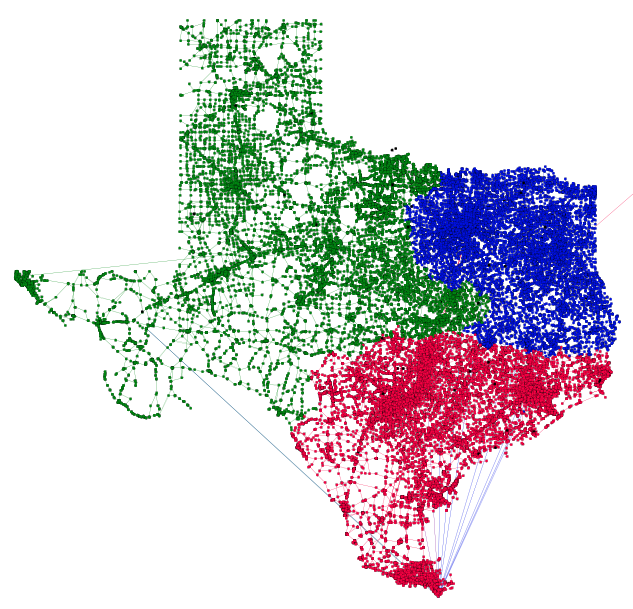}}}
\caption{Two- and Three-subnetwork partitions for Texas that performed the best relative to the centralized gradient projection} \label{fig:allpartTexas}
\end{figure}

For the Texas network, we run DSTAP-Heuristic for one iteration and use the obtained solution to warm start the gradient projection solver. We compare the computation time taken to solve the megaregion to a relative gap less than 1E$-4$ compared relative to the centralized method without warmstarting. We observe that the solution time for TAP using centralized gradient projection without warmstarting in 384 minutes. Whereas, if we allow warmstarting through DSTAP-Heuristic, then the solution time using two subnetworks is 269 minutes and using three subnetworks is 261 minutes. Thus, we save 30\% of computation time on average using DSTAP-heuristic to warm start the TAP for megaregions.

\begin{figure}[H]
\centering
\noindent\makebox[\textwidth]{%
\resizebox*{8.5cm}{!}{\includegraphics{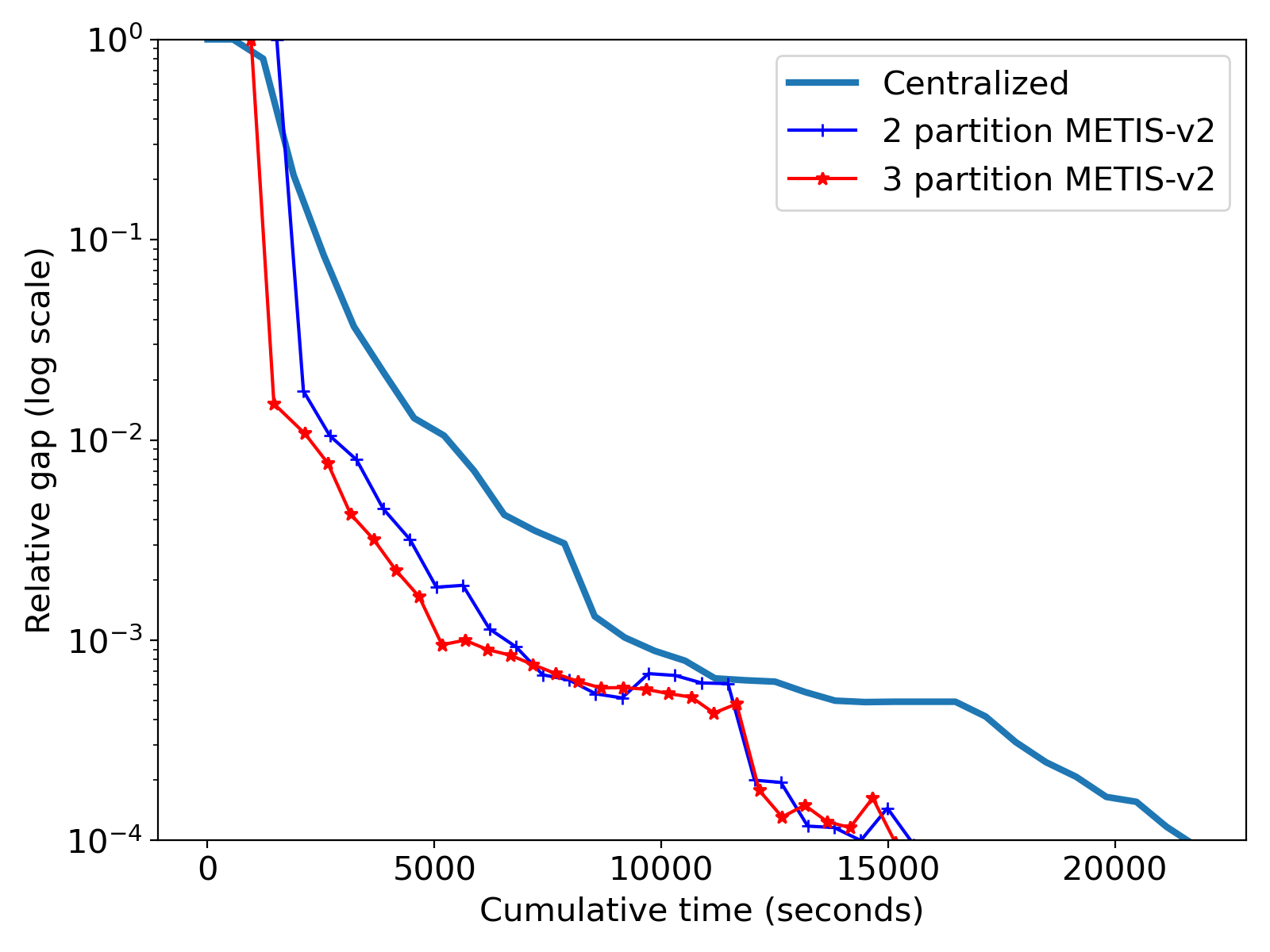}}}
\centering
\caption{Performance on the Texas network using METISv2 two and three subnetwork partitions }
\label{fig:Texas_res}
\centering
\end{figure}

\section{Conclusions}
\label{sec:conclusions}

Computational efficiency of TAP is an important concern as networks grow in size and resolution and for megaregional planning. This article proposes a network decomposition heuristic for solving TAP with associated network transformation for solving the constrained shortest path and an iterative refinement approach for network partitioning. The results compare the performance of the proposed heuristic and centralized gradient projection algorithm for multiple large-scale networks, showing the computational benefits of warmstarting using the proposed heuristic, with an average computational savings of 10–35\% over a TAP solver without
warmstarting. The computational benefit is observed for all tested networks and for different demand levels. Partitions generated by METISv2 algorithm or its variant obtained after $\psi$-FM refinement were found to achieve the lowest gap after an iteration of DSTAP-Heuristic.  Based on our findings, we recommend the use of DSTAP-Heuristic to warm start TAP solutions for finding efficient solutions in limited time. 

Various directions for future research can be identified. First, investigating newer network partitioning algorithms with additional metrics or explanatory variables can help improve the performance of DSTAP-Heuristic. This includes partitioning  based on latitude and longitudes of nodes in a regional network~\citep{schild2015balanced} and/or natural divisions, such as a river, that have demonstrated efficiency for solving shortest path on megaregional networks. Second, various component models for DSTAP-Heuristic can be enhanced based on latest advancements in the TAP literature. These include using contraction hierarchies~\citep{schneck2020accelerating}, using bush-based algorithms for solving TAP on master and subnetworks~\citep{bar2010traffic}, and incorporating other techniques of parallelization such as ones based on OD pair~\citep{chen1988parallel} or parallel block-coordinate-descent for gradient projection~\citep{chen2020parallel} in tandem with geographic partitions. Finally, the geographical partitioning methods can be extended for solving equilibrium on dynamic networks, which pose additional constraints due to time dependencies and modeling of queue spillbacks.




\section*{Acknowledgements}
Authors would like to thank Texas Department of Transportation for sharing the statewide model data that was used to create the Texas network used in this analysis. We also acknowledge the input and feedback from Steve Boyles on several earlier drafts of this article. Partial research by the first author was also supported by the Center for Advanced Transportation Mobility, University Transportation Center, and NSF Grant No. 2106989.






\bibliographystyle{tfcad}
\bibliography{interactcadsample}

\clearpage

\appendix

\section{Proof for Lemma~\ref{lem:part2}}
\label{appendix:proof}
\begin{proof}
First, we define a mapping relating every path in the transformed network to a path in the original network. Let $\Pi_{\text{transformed}}$ be set of all paths in the transformed network and $\Pi_{\text{original}}^{\text{feasible}}$ be the set of all feasible paths in the original network. We define a mapping $f:\Pi_{\text{transformed}} \rightarrow \Pi_{\text{original}}^{\text{feasible}}$ as follows:
    \begin{itemize}
        \item If the path passes through a node $z_o$, $z_d$, $z_p$ or $z_a$, replace the occurrence with the $z$
        \item If node $z$ has multiple consecutive occurrences, then replace it with one occurrence of $z$
    \end{itemize}

For example, a path $[1_o, 1_p, 2_a, 3, 4_p, 5_p, 6, 7_a, 7_d]$ will map to a path $[1,2,3,4,5,6,7]$ in the original network. We can now show that $f$ is a surjective mapping which completes the proof.

Let $[o,x^1,\ldots,x^n,d] \in \Pi_{\text{original}}^{\text{feasible}}$ be a feasible path representing an ordered list of nodes connecting origin zone $o$ to destination zone $d$. For the mapping $f$ to be surjective, we must be able to find a path $\pi \in \Pi_{\text{transformed}}$ that maps to this path. In a general case, this path may pass through other centroid nodes. Let's assume that all nodes $x^1,x^2,\ldots,x^n$ are centroids. Then, by the construction of network transformation, if the link connecting $(x^i,x^{i+1})$ is artificial, then we have replaced it with link $(x^i_p, x^{i+1}_a)$. On the other hand, if the link connecting $(x^i,x^{i+1})$ is a physical link, then we have replaced it with two links $(x^i_a, x^{i+1}_p)$ and $(x^i_p, x^{i+1}_p)$.   
This process applies for all links in the original network path, hence, there exists a path in the transformed network for every feasible path in the original network.
\end{proof}

    
    




\clearpage

\section{Detailed Partition Statistics}
\label{appendix:partition_stats}
Austin network results are shown in Table~\ref{tab:appendixAustinResult}.\footnote{The total number of links (or nodes) across different partitions along with the cut links (or boundary nodes) may not add up to the total in the network as certain nodes/links were dropped to ensure network connectivity and those nodes/links were not being used.}

\begin{table}[h]
\centering
\caption{Austin Network Partition}
\label{tab:appendixAustinResult}
\begin{tabular}{|c|c|c|c|}
\hline
\multicolumn{4}{|c|}{\cellcolor[HTML]{C0C0C0}\textbf{Austin network ($n=7466$, $m=18710$)}} \\ \hline
Partition & METISv1 & \begin{tabular}[c]{@{}c@{}}$\psi-$FM\\ METISv1\end{tabular} & \begin{tabular}[c]{@{}c@{}}$\psi-$FM\\ METISv2\end{tabular} \\ \hline
$n_1:n_2$ & 3796 : 3667 & 3799 : 3664 & 3685 : 3779 \\ \hline
$m_1:m_2$ & 10434 : 8169 & 9387 : 9207 & 9242 : 9376 \\ \hline
\begin{tabular}[c]{@{}c@{}}\# boundary\\ nodes\end{tabular} & 94 & 98 & 90 \\ \hline
\begin{tabular}[c]{@{}c@{}}\# cut\\ links\end{tabular} & 95 & 101 & 81 \\ \hline
$\psi$ & 9580.93 & 6810.11 & 16223.7 \\ \hline
\begin{tabular}[c]{@{}c@{}}Centralized\\ time to reach\\ gap 1E-4\end{tabular} & 502.72 & 502.72 & 502.72 \\ \hline
\begin{tabular}[c]{@{}c@{}}Total DSTAP\\ warmstarting\\ time (sec)\end{tabular} & 448.66 & 410.68 & 416.46 \\ \hline
\begin{tabular}[c]{@{}c@{}}Fullnet\\ gap from 1\\ itr of DSTAP\end{tabular} & 0.1698 & 0.1312 & 0.1712 \\ \hline
\end{tabular}
\end{table}

\begin{sidewaystable}[]
\centering
\caption{Berlin-center Network Partition}
\label{tab:appendixBerlinResult}
\begin{tabular}{|c|c|c|c|c|c|}
\hline
\multicolumn{6}{|c|}{\cellcolor[HTML]{C0C0C0}\textbf{Berlin-center network ($n=12981$, $m=28376$)}} \\ \hline
\textbf{Partition} & \textbf{Spectral} & \textbf{\begin{tabular}[c]{@{}c@{}}$\psi-$FM\\ Spectral\end{tabular}} & \textbf{METISv1} & \textbf{METISv2} & \textbf{\begin{tabular}[c]{@{}c@{}}$\psi-$FM\\ METISv2\end{tabular}} \\ \hline
$n_1:n_2$ & 6443 : 6538 & 6450 : 6531 & 5891 : 5794 & 6432 : 6549 & 6427 : 6554 \\ \hline
$m_1:m_2$ & 13590 : 14634 & 13610 : 14593 & 13111 : 12966 & 13527 : 14730 & 13498 : 14743 \\ \hline
\begin{tabular}[c]{@{}c@{}}\# boundary\\ nodes\end{tabular} & 170 & 183 & 217 & 127 & 135 \\ \hline
\begin{tabular}[c]{@{}c@{}}\# cut\\ links\end{tabular} & 146 & 167 & 231 & 113 & 129 \\ \hline
$\psi$ & 4700.74 & 3789.65 & 995.5 & 3429.92 & 2934.01 \\ \hline
\begin{tabular}[c]{@{}c@{}}Centralized\\ time to reach\\ gap 1E-4 (sec)\end{tabular} & \multicolumn{5}{c|}{149.79} \\ \hline
\begin{tabular}[c]{@{}c@{}}Total DSTAP\\ warmstarting\\ time (sec)\end{tabular} & 78.242 & 87.56 & 63.41 & 77.44 & 77.91 \\ \hline
\begin{tabular}[c]{@{}c@{}}Fullnet\\ gap from 1\\ itr of DSTAP\end{tabular} & 0.0535 & 0.0289 & 0.0161 & 0.0273 & 0.0272 \\ \hline
\end{tabular}
\end{sidewaystable}

\begin{sidewaystable}[]
\centering
\caption{Goldcoast Network Partition}
\label{tab:appendixGoldcoastResult}
\begin{tabular}{|c|c|c|c|c|c|c|}
\hline
\multicolumn{7}{|c|}{\cellcolor[HTML]{C0C0C0}\textbf{Goldcoast network ($n=4807$, $m=11140$)}} \\ \hline
\textbf{Partition} & \textbf{Spectral} & \textbf{\begin{tabular}[c]{@{}c@{}}$\psi-$FM\\ Spectral\end{tabular}} & \textbf{METISv1} & \textbf{\begin{tabular}[c]{@{}c@{}}$\psi-$FM\\ METISv1\end{tabular}} & \textbf{METISv2} & \textbf{\begin{tabular}[c]{@{}c@{}}$\psi-$FM\\ METISv2\end{tabular}} \\ \hline
$n_1:n_2$ & 2194 : 2587 & 2194 : 2587 & 2421 : 2356 & 2419 : 2358 & 2397 : 2384 & 2396 : 2385 \\ \hline
$m_1:m_2$ & 5108 : 6008 & 5108 : 6008 & 5619 : 5474 & 5615 : 5479 & 5518 : 5598 & 5516 : 5599 \\ \hline
\begin{tabular}[c]{@{}c@{}}\# boundary\\ nodes\end{tabular} & 23 & 23 & 35 & 34 & 22 & 23 \\ \hline
\begin{tabular}[c]{@{}c@{}}\# cut\\ links\end{tabular} & 20 & 20 & 33 & 32 & 20 & 21 \\ \hline
$\psi$ & 241.92 & 241.92 & 929.43 & 415.63 & 2803.21 & 2803.21 \\ \hline
\begin{tabular}[c]{@{}c@{}}Centralized\\ time to reach\\ gap 1E-4 (sec)\end{tabular} & \multicolumn{6}{c|}{615.11} \\ \hline
\begin{tabular}[c]{@{}c@{}}Total DSTAP\\ warmstarting\\ time (sec)\end{tabular} & 545.85 & 555.12 & 545.05 & 625.9 & 563.03 & 513.99 \\ \hline
\begin{tabular}[c]{@{}c@{}}Fullnet\\ gap from 1\\ itr of DSTAP\end{tabular} & 0.0607 & 0.0598 & 0.069 & 0.068 & 0.0682 & 0.075 \\ \hline
\end{tabular}
\end{sidewaystable}

\begin{sidewaystable}[]
\centering
\caption{Chicago-regional Network Partition}
\label{tab:appendixChicagoResult}
\begin{tabular}{|c|c|c|c|c|c|c|}
\hline
\multicolumn{7}{|c|}{\cellcolor[HTML]{C0C0C0}\textbf{Chicago regional network ($n=12982$, $m=39018$)}} \\ \hline
\textbf{Partition} & \textbf{Spectral} & \textbf{\begin{tabular}[c]{@{}c@{}}$\psi-$FM\\ Spectral\end{tabular}} & \textbf{METISv1} & \textbf{\begin{tabular}[c]{@{}c@{}}$\psi-$FM\\ METISv1\end{tabular}} & \textbf{METISv2} & \textbf{\begin{tabular}[c]{@{}c@{}}$\psi-$FM\\ METISv2\end{tabular}} \\ \hline
$n_1:n_2$ & 6399 : 6580 & 6405 : 6574 & 6562 : 6411 & 6555 : 6418 & 6352 : 6627 & 6351 : 6628 \\ \hline
$m_1:m_2$ & 19136 : 19675 & 19145 : 19673 & 19289 : 19425 & 19277 : 19440 & 19240 : 19627 & 19229 : 19630 \\ \hline
\begin{tabular}[c]{@{}c@{}}\# boundary\\ nodes\end{tabular} & 192 & 195 & 246 & 241 & 152 & 153 \\ \hline
\begin{tabular}[c]{@{}c@{}}\# cut\\ links\end{tabular} & 207 & 200 & 288 & 285 & 151 & 159 \\ \hline
$\psi$ & 81512.19 & 48238.04 & 12598.71 & 12598.71 & 33672.62 & 32821.14 \\ \hline
\begin{tabular}[c]{@{}c@{}}Centralized\\ time to reach\\ gap 1E-4 (sec)\end{tabular} & \multicolumn{6}{c|}{1668.28} \\ \hline
\begin{tabular}[c]{@{}c@{}}Total DSTAP\\ warmstarting\\ time (sec)\end{tabular} & 1558.96 & 1558.69 & 4696.46 & 1568.57 & 1544.49 & 1680.52 \\ \hline
\begin{tabular}[c]{@{}c@{}}Fullnet\\ gap from 1\\ itr of DSTAP\end{tabular} & 0.0623 & 0.0611 & 0.1077 & 0.0955 & 0.0664 & 0.0744 \\ \hline
\end{tabular}
\end{sidewaystable}

\begin{sidewaystable}[]
\centering
\caption{Philadelphia Network Partition}
\label{tab:appendixPhiladelphiaResult}
\begin{tabular}{|c|c|c|c|c|c|c|}
\hline
\multicolumn{7}{|c|}{\cellcolor[HTML]{C0C0C0}\textbf{Philadelphia network ($n=13389$, $m=40003$)}} \\ \hline
\textbf{Partition} & \textbf{Spectral} & \textbf{\begin{tabular}[c]{@{}c@{}}$\psi-$FM\\ Spectral\end{tabular}} & \textbf{METISv1} & \textbf{\begin{tabular}[c]{@{}c@{}}$\psi-$FM\\ METISv1\end{tabular}} & \textbf{METISv2} & \textbf{\begin{tabular}[c]{@{}c@{}}$\psi-$FM\\ METISv2\end{tabular}} \\ \hline
$n_1:n_2$ & 4867 : 8522 & 4863 : 8526 & 6858 : 6531 & 6852 : 6537 & 6734 : 6655 & 6728 : 6661 \\ \hline
$m_1:m_2$ & 14751 : 25130 & 14743 : 25147 & 19798 : 19929 & 19773 : 19952 & 20624 : 19223 & 20606 : 19239 \\ \hline
\begin{tabular}[c]{@{}c@{}}\# boundary\\ nodes\end{tabular} & 104 & 107 & 259 & 252 & 152 & 152 \\ \hline
\begin{tabular}[c]{@{}c@{}}\# cut\\ links\end{tabular} & 122 & 113 & 276 & 278 & 156 & 158 \\ \hline
$\psi$ & 459897.3 & 359189.24 & 367387.31 & 281971.62 & 310871.26 & 278558.5 \\ \hline
\begin{tabular}[c]{@{}c@{}}Centralized\\ time to reach\\ gap 1E-4 (sec)\end{tabular} & \multicolumn{6}{c|}{2403.5} \\ \hline
\begin{tabular}[c]{@{}c@{}}Total DSTAP\\ warmstarting\\ time (sec)\end{tabular} & 2354.91 & 2331.33 & 2689.24 & 2418.92 & 2469.81 & 2351.194 \\ \hline
\begin{tabular}[c]{@{}c@{}}Fullnet\\ gap from 1\\ itr of DSTAP\end{tabular} & 0.2143 & 0.0995 & 0.2937 & 0.2260 & 0.8109 & 0.7226 \\ \hline
\end{tabular}
\end{sidewaystable}

\end{document}